\documentclass[11pt,reqno,a4paper]{amsart}
\usepackage{a4wide,verbatim}

\title[Equivariant minimal surfaces]{Equivariant minimal surfaces in $\CH^2$ and their Higgs bundles.}
\author{John Loftin}
\address{Department of Mathematics and Computer Science\\ Rutgers-Newark\\
Newark, NJ 07102, USA}
\email{loftin@rutgers.edu}

\author{Ian McIntosh}
\address{Department of Mathematics\\ University of York\\ 
York YO10 5DD, UK}
\email{ian.mcintosh@york.ac.uk}
\date{March 12, 2018}

\newcommand{\Z}{\mathbb{Z}}

\newcommand{\C}{\mathbb{C}}

\newcommand{\R}{\mathbb{R}}

\newcommand{\CP}{\mathbb{CP}}
\newcommand{\CH}{\mathbb{CH}}
\newcommand{\RH}{\mathbb{RH}}
\newcommand{\RP}{\mathbb{RP}}
\newcommand{\CdS}{\mathbb{C}\mathrm{dS}}
\renewcommand{\P}{\mathbb{P}}

\newcommand{\caC}{\mathcal{C}}
\newcommand{\caD}{\mathcal{D}}
\newcommand{\caE}{\mathcal{E}}
\newcommand{\caF}{\mathcal{F}}

\newcommand{\caH}{\mathcal{H}}

\newcommand{\caL}{\mathcal{L}}

\newcommand{\caN}{\mathcal{N}}
\newcommand{\caO}{\mathcal{O}}
\newcommand{\caP}{\mathcal{P}}
\newcommand{\caQ}{\mathcal{Q}}
\newcommand{\caR}{\mathcal{R}}

\newcommand{\caT}{\mathcal{T}}
\newcommand{\caU}{\mathcal{U}}
\newcommand{\caV}{\mathcal{V}}
\newcommand{\caW}{\mathcal{W}}

\newcommand{\fF}{\mathfrak{F}}

\newcommand{\su}{\mathfrak{su}}

\newcommand{\fg}{\mathfrak{g}}

\newcommand{\fh}{\mathfrak{h}}
\newcommand{\fm}{\mathfrak{m}}

\newcommand{\fu}{\mathfrak{u}}
\newcommand{\fz}{\mathfrak{z}}

\newcommand{\End}{\operatorname{End}}

\newcommand{\Hom}{\operatorname{Hom}}

\newcommand{\Jac}{\operatorname{Jac}}
\newcommand{\Pic}{\operatorname{Pic}}
\newcommand{\Isom}{\operatorname{Isom}}

\renewcommand{\Re}{\operatorname{Re}}
\renewcommand{\Im}{\operatorname{Im}}

\newcommand{\Kah}{K\" ahler\ }

\newcommand{\Ad}{\operatorname{Ad}}

\newcommand{\tr}{\operatorname{tr}}
\newcommand{\rk}{\operatorname{rk}}

\newcommand{\g}[2]{\langle{#1},{#2}\rangle}

\newcommand{\II}{\mathrm{I\! I}}

\newcommand{\Ric}{\operatorname{Ric}}

\newcommand{\grad}{\operatorname{grad}}
\newcommand{\Area}{\operatorname{Area}}
\newcommand{\vol}{v}
\newcommand{\0}{\mathbf{0}}
\newcommand{\Hess}{\operatorname{Hess}}

\newtheorem{thm}{Theorem}[section]
\newtheorem{prop}[thm]{Proposition}
\newtheorem{lem}[thm]{Lemma}
\newtheorem{cor}[thm]{Corollary}

\newtheorem{defn}[thm]{Definition}
\theoremstyle{remark}

\newtheorem{rem}{Remark}[section]

\numberwithin{equation}{section}

\begin{document}

\begin{abstract}
This paper gives a construction for all minimal immersions $f$ of the Poincar\'{e} disc into the complex hyperbolic
plane $\CH^2$  which are equivariant with respect to an irreducible representation $\rho$ of a hyperbolic surface group 
into $PU(2,1)$. We 
exploit the fact that each such immersion is a twisted conformal harmonic map and therefore has a corresponding 
Higgs bundle. We identify the structure of these Higgs bundles and show how each is determined by properties of the map, including 
the induced metric and a holomorphic cubic differential on the surface.
We show that the moduli space of pairs $(\rho,f)$ is a disjoint union of finitely many complex manifolds, whose structure 
we fully describe. The holomorphic (or anti-holomorphic) maps provide multiple components of this union, as do the 
non-holomorphic maps. Each of the latter components has the same dimension as the
representation variety for $PU(2,1)$, and is indexed by the number of complex and anti-complex points of the immersion.
These numbers determine the Toledo invariant and the Euler number of the normal bundle of the immersion.
We show that there is an open set of quasi-Fuchsian representations of Toledo
invariant zero for which the minimal surface is unique and Lagrangian.
\end{abstract}

\maketitle

\section{Introduction}
In this article we provide a complete classification of $\rho$-equivariant minimal immersions $f:\caD\to\CH^2$ and a
parametrisation for their moduli space as a union of complex manifolds. Here $\caD$ is the Poincar\'{e} disc and $\rho$ is an 
irreducible (or more generally, reductive) representation of a 
hyperbolic surface group (i.e., the fundamental group $\pi_1\Sigma$ of a closed orientable surface
$\Sigma$ of genus at least two) into the group $PU(2,1)=U(2,1)/\text{centre}$. 
Recall that a minimal immersion of a surface is the same thing as a conformal harmonic map. 
To say $f$ is $\rho$-equivariant means it intertwines the action of a Fuchsian group on $\caD$ with the action of 
$\rho$ on the complex hyperbolic plane $\CH^2$ by holomorphic isometries. 
One can also think of $f$ as a section of a $\CH^2$-bundle over $\Sigma$ or, when $\rho$ is a discrete
embedding, as a minimal immersion $f:\Sigma\to \CH^2/\rho$ into the quotient manifold. 

We classify these pairs $(\rho,f)$ up to $PU(2,1)$-equivalence, i.e., up to the natural
left action of $PU(2,1)$ by conjugation of $\rho$ and the simultaneous ambient
isometry of $f$. In particular, the space of such pairs has a natural ``forgetful'' map to the moduli space
\[
\caR(G) = \Hom^+(\pi_1\Sigma,G)/G,
\]
of conjugacy classes of reductive representations into $G=PU(2,1)$. Recall that
this is a real analytic variety whose connected components are indexed by the
Toledo invariant $\tau(\rho)\in\tfrac23\Z$, $|\tau(\rho)|\leq -\chi(\Sigma)$ (see, e.g., \cite{Gol}). We show that
there are families of pairs for every value of $\tau$.

We achieve this classification by exploiting 
the powerful machinery which links equivariant harmonic maps to the Yang-Mills-Higgs 
equations over a compact Riemann surface and thereby to Higgs bundles \cite{Don,Hit87,Cor,Sim}, 
Our starting point is 
two facts: 
(i) to each irreducible $\rho$ and marked
conformal structure on $\Sigma$ (i.e., conjugacy class of Fuchsian representations) there is a $\rho$-equivariant 
harmonic map $f$ \cite{Don,Cor}; (ii) when this information is encoded into a $G$-Higgs bundle $(E,\Phi)$ 
$f$ is weakly conformal precisely when $\tr\Phi^2=0$. It is surprising that this approach has not, to our knowledge,
been exploited before, since it provides a very effective way to understand the moduli as holomorophic data and avoids
a direct analysis of the Gauss-Codazzi equations for minimal surfaces (cf.\ \cite{Tau}).

When $G=PU(2,1)$ the structure of the Higgs bundles and their moduli space is quite well understood \cite{Xia,Got,BraGG}. 
In this case the bundle $E$ splits into a sum $V\oplus L$
of a rank two sub-bundle $V$ and a rank one sub-bundle $L$ which are mapped to each other by the Higgs field, i.e.,
we can write $\Phi=(\Phi_1,\Phi_2)$ where $\Phi_1:L\to KV$ and $\Phi_2:V\to KL$ for $K$ the canonical bundle determined
by the marked conformal structure. In fact by projective equivalence we may assume that $L=1$, the trivial
bundle. The Higgs field $\Phi$ corresponds to the differential of $f$. To be
precise, it corresponds to 
\[
\partial f:T^{1,0}\caD\to T^\C\CH^2,
\]
and the components $\Phi_1,\Phi_2$ correspond to the components of $\partial f$ with respect to the type decomposition of
$T^\C\CH^2$. For a minimal surface there are two possibilities: (i) $f$ is holomorphic ($\Phi_2=0$) or anti-holomorphic
($\Phi_1=0$), or (ii) $f$ is neither and has isolated complex and anti-complex points which give finite divisors $D_2$ and
$D_1$ over $\Sigma$ (where, respectively, $\Phi_2$ and $\Phi_1$ have zeroes). We treat these two possibilities separately,
and in fact it is the latter which we treat first, in \S 2 and \S 3. An important role is played by a cubic
holomorphic differential $\caQ$ which can be naturally assigned to any minimal surface in a \Kah manifold of constant
holomorphic section curvature \cite{Woo84}. It vanishes identically for holomorphic or anti-holomorphic immersions 
(but not only for them). When $f$ is neither holomorphic nor anti-holomorphic
we show that, with the two bilinear forms $\gamma_j=\tfrac12\tr\Phi_j\Phi_j^\dagger$, which carry the information of the
metric $\gamma$ induced by $f$, the data $(\gamma_1,\gamma_2,\caQ)$ completely determines the minimal immersion up to
ambient isometries. The principal results of \S 2 and \S 3 (Theorems \ref{thm:minimal} and \ref{thm:xi})
can be summarised as follows.
\begin{thm}\label{thm:nonhol}
Let $\rho$ be irreducible and $f$ be minimal and $\rho$-equivariant. If $f$ is neither holomorphic nor anti-holomorphic
then the pair $(\rho,f)$ is faithfully determined, up to $G$-equivalence, by data $(\Sigma_c,D_1,D_2,\xi)$ where: $c$ is a
marked conformal structure on $\Sigma$, $D_1$ and $D_2$ are effective divisors on $\Sigma$ whose degrees $d_1,d_2$ satisfy 
\[
2d_1+d_2<6(g-1) \text{ and } d_1+2d_2<6(g-1),
\]
and $\xi\in H^1(\Sigma_c,K^{-2}(D_1+D_2))$ represents an extension class which determines $V$ as an extension of $K^{-1}(D_1)$
by $K(-D_2)$. The Higgs bundle is then $E=V\oplus 1$ equipped with a Higgs field determined by the extension class. Under 
the Dolbeault isomorphism this extension class corresponds to the cohomology class of
$-\bar\caQ/\gamma_1\gamma_2$ in $H^{0,1}(\Sigma_c,K^{-2}(D_1+D_2))$, and $\xi=0$ if and only if $\caQ=0$. In this
correspondence, $\rho$ has Toledo invariant $\tau(\rho)=\tfrac23(d_2-d_1)$, and the Euler number of the normal bundle of
$f$ is $\chi(T\Sigma^\perp) = 2(g-1) -d_1-d_2$.
\end{thm}
In particular, the integers $d_1,d_2$ determine, and are determined by, the Toledo invariant of
$\rho$ and the Euler number of the normal bundle of $f$.

The correspondence is constructive in the sense that, given data $(\Sigma_c,D_1,D_2,\xi)$ satisfying the conditions above,
we give an explicit construction of a stable Higgs bundle $(E,\Phi)$ with $\tr\Phi^2=0$.
The divisors $D_1,D_2$ are precisely the divisors of anti-complex and complex points of the minimal immersion. For $f$ to be
strictly an immersion they must have no common points, but the construction still works to produce branched minimal
immersions if they intersect, with branch points at the intersections. 

This construction includes and greatly extends the construction of minimal Lagrangian embeddings we gave in
\cite{LofM13}. Indeed, up to this time we were not aware of any other examples of non-holomorphic equivariant 
minimal immersions into $\CH^2$ (save for the Lagrangian examples whose existence is a consequence of the ``mountain-pass''
solutions to the Gauss equation described in \cite{HuaLL}). The theorem above not only gives all
examples for reductive representations (when we allow branch points) but allows us to describe the structure of the 
moduli space of these as a complex manifold (Theorem \ref{thm:moduli} below).

By a theorem of Wolfson \cite{Wol89}, $f$ is Lagrangian precisely when it has no complex or anti-complex points,
and is therefore parametrised by the pair $(\Sigma_c,\xi)$. 
The embeddings constructed in \cite{LofM13} all have the property
that the exponential map on the normal bundle provides a diffeomorphism between $T\caD^\perp$ and $\CH^2$. It follows that
$\rho$ has a finite fundamental
domain (given by the normal bundle over a finite fundamental domain for the action of $\pi_1\Sigma$
on $\caD$). Here we call such embeddings \emph{almost $\R$-Fuchsian}, because they are deformations of
the embedding $\RH^2\to\CH^2$, which is equivariant with respect to every Fuchsian representation into $SO(2,1)\subset
PU(2,1)$. In \S 4 we improve on the results in \cite{LofM13} by showing that whenever $f$ is 
minimal Lagrangian with
$\|\caQ\|^2_\gamma< 2$ it is almost $\R$-Fuchsian and the unique $\rho$-equivariant minimal immersion.
The $\R$-Fuchsian case corresponds to $\caQ=0$, and therefore $\xi=0$ by the theorem above.  
The uniqueness of $f$ proved here implies that the
data $(\Sigma_c,\xi)$ also parametrises the almost $\R$-Fuchsian family, although at present 
we do not understand the appropriate bound on $\xi$. It is preferable to have the parametrisation in terms of
$(\Sigma_c,\xi)$ since that gets us directly to the Higgs bundle and therefore to $\rho$. The parametrisation in
\cite{LofM13} using $\caQ$ requires an additional condition to provide a unique solution to the Gauss
equation of the immersion. We describe the subtleties of existence and uniqueness for this equation in 
\S \ref{sec:Gauss}, which also draws on earlier work by Huang, Loftin \& Lucia \cite{HuaLL}.

The minimal Lagrangian case suggests that it is important to understand those minimal immersions for which $\caQ=0$. These
are treated in \S 5. We show that they
have a very interesting interpretation in terms of the Higgs bundle, for $\caQ=0$ exactly when the Higgs bundle is a
\emph{Hodge bundle} (or \emph{variation of Hodge structure}). These are known to be the critical points of the Morse-Bott
function $\|\Phi\|^2_{L^2}$ \cite{Got} and come in two flavours: length two or length three. The length-two Hodge bundles
all correspond to holomorphic or anti-holomorphic maps. The following theorem summarises our results
regarding these.
\begin{thm}\label{thm:introhol}
Let $\rho$ be irreducible and $f$ be branched holomorphic and $\rho$-equivariant. Then the pair $(\rho,f)$ is faithfully 
determined by  data $(\Sigma_c,B,L,\C.\eta)$ where $B$ is an effective divisor of degree $b$, $L$ is a holomorphic 
line bundle of degree $l$ satisfying
\[
3(g-1)+\tfrac12 b<l<6(g-1)-b,\quad 0\leq b<2(g-1),
\]
and $\C.\eta\in \P H^1(\Sigma_c,KL^{-1}(B))$ determines the isomorphism class of $V$ as an extension of $K^{-1}(B)$ by $K^{-2}L$. 
The Toledo invariant of $\rho$ is $\tfrac23 (6g-6-b-l)>0$. 
\end{thm}
This also accounts for anti-holomorphic immersions, since $f$ is anti-holomorphic and $\rho$-equivariant if and only if 
$\bar f$ is holomorphic and $\bar\rho$-equivariant.  The latter has the dual Higgs bundle to $\rho$.
But unlike Theorem \ref{thm:nonhol} these conditions are only necessary conditions on the data to produce a
stable Higgs bundle: this is explained in more detail in \S \ref{sec:Q=0}.

As with the non-holomorphic case, the extension class $\eta$ corresponds to the Dolbeault cohomology class of a tensor over 
$\Sigma_c$ which has
geometrical significance and is related to the second fundamental form of $f$ (see Theorem \ref{thm:eta}). We explain how
the limiting value $\eta=0$ corresponds to reducible representations which are not \emph{maximal}, i.e., 
do not have $\tau=\pm\chi(\Sigma)$.

By contrast, the length-three Hodge bundles correspond to those pairs $(\rho, f)$ coming from Theorem \ref{thm:nonhol} 
with $\xi=0$. By using the method of harmonic sequences \cite{BolW,CheW83,EelW,ErdG} we show that $\xi=0$ precisely when
the harmonic sequence of $f$ contains a holomorphic $\rho$-equivariant (and ``timelike'') map into complex de Sitter 
$2$-space. This is a pseudo-Hermitian symmetric space, the 
analogue of the real $2$-dimensional de Sitter space which complements $\RH^2$.

In the final section, \S 6, we describe the moduli space 
\[
\caV = \{(\rho,f):\text{$\rho$ irreducible, $f$ branched minimal}\}/G.
\]
By the results stated above it is a union of components $\caV(d_1,d_2)$ containing those pairs described by Theorem
\ref{thm:nonhol} and $\caW_\tau$ containing all holomorphic or anti-holomorphic maps with Toledo invariant $\tau$.
Of these we prove: 
\begin{thm}\label{thm:moduli}
Each $\caV(d_1,d_2)$ is a complex manifold of dimension $8g-8$, while 
$\caW_\tau$ is a complex manifold of dimension $9(g-1)-\tfrac{3}{2}\tau$. Each
$\caV(d_1,d_2)$ is diffeomorphic to a bundle over the Teichm\"{u}ller space of $\Sigma$, and the 
fibres are complex analytic submanifolds.  Each fibre $\caV_c(d_1,d_2)$ is a rank $5g-5-d_1-d_2$ vector bundle 
over $S^{d_1}\Sigma_c\times
S^{d_2}\Sigma_c$. For $\caW_\tau$ each fibre $\caW_\tau(c)$ is 
birational to a $\CP^N$-bundle over
$\Jac(\Sigma_c)$ where $N=5(g-1)-\tfrac{3}{2}\tau -1$.
\end{thm}
We finish in \S \ref{sec:R(G)} with a brief discussion of the map $\caV\to \caR(G)$ given by forgetting the immersion.
There is much
yet to be understood about this map. For example, we do not know if this map is onto for non-maximal representations, 
or the dimension of its image on components of $\caV$. Nevertheless, we can make some 
salient remarks about its restriction to any fibre over 
Teichm\"{u}ller space. We point out that on the image of $\caV$ the $L^2$-norm of the Higgs
fields equals the area of the minimal immersion. The critical points of $\|\Phi\|^2_{L^2}$ are all accounted for by
the Hodge bundles, and therefore lie in the image of $\caV$. A comparison of the structure of $\caV_c(d_1,d_2)$ with what
is known of the Morse index from \cite{Got}, together with an area bound established in \S 3, suggests that the 
fibres of the vector bundle $\caV_c(d_1,d_2)$ map onto the downward Morse flow of $\|\Phi\|^2_{L^2}$.

For us, one of the outstanding challenges is to use this construction to study the \emph{quasi-Fuchsian}
representations, where ``quasi-Fuchsian'' is meant in
the sense of Parker \& Platis \cite{ParP10}, i.e., a convex cocompact, totally loxodromic, discrete embedding. 
Recent work of Guichard \& Wienhard \cite[Thm 1.8]{GuiW} has shown that $\rho$ is a convex cocompact embedding 
precisely when it is an \emph{Anosov embedding}. Since the latter are totally
loxodromic, the notions ``quasi-Fuchsian'', ``convex cocompact embedding'' and ``Anosov embedding'' 
coincide for $PU(2,1)$ (and more generally, any semisimple
Lie group of real rank one). One knows from \cite{Gui} that quasi-Fuchsian representations comprise
an open subset of the representation variety $\caR(G)$.  Examples for
every even value of $\tau(\rho)$ were constructued by Goldman, Kapovich \& Leeb \cite{GolKL}, while Parker \& Platis 
\cite{ParP06} constructed a open family of quasi-Fuchsian representations in the Toledo invariant 
zero component. The latter family are perturbations of the Fuchsian representations corresponding to the totally geodesic
and Lagrangian embedding $\RH^2\subset \CH^2$, so they must overlap with the almost $\R$-Fuchsian representations constructed
in \cite{LofM13} (all of which are quasi-Fuchsian). 

Beyond these examples, there is very little known;
it is not even known whether there are any quasi-Fuchsian representations for non-integral 
values of the Toledo invariant. One compelling reason for looking to minimal immersions to provide
more insight is the theorem of Goldman \& Wentworth \cite{GolW}, that for convex cocompact
representations the harmonic map energy functional on Teichm\"{u}ller space is a
proper function.  It therefore has at least one critical point, and it is a well-known result of
Sacks \& Uhlenbeck \cite{SacU} that each critical point
corresponds to a weakly conformal harmonic (i.e., branched minimal) map. Since our
construction includes branched minimal maps, it follows that the map
$\caV\to\caR(G)$ has all quasi-Fuchsian representations in its image.

\smallskip\noindent
\textbf{Acknowledgements.} 
The authors gratefully acknowledge support from U.S. National Science Foundation grants DMS 1107452, 1107263, 1107367
\emph{RNMS: Geometric Structures and Representation Varieties} (the GEAR Network). The first author gratefully acknowledges
support support from a Simons Collaboration Grant for Mathematicians 210124. The first author is grateful to Olivier
Guichard, Zeno Huang, and Marcello Lucia for useful conversations. The second author thanks Peter Gothen for his remarks
regarding the Morse flow on the moduli space of Higgs bundles.

\section{Minimal surfaces and their Higgs bundles.}

We begin by setting up the
notation and standard constructions for the minimal surfaces and the Higgs bundles we will be working with.

\subsection{Equivariant minimal surfaces in $\CH^2$.}
Our model for $\CH^2$ will be the projective model, as follows.
Let $\C^{2,1}$ denote the vector space $\C^3$ equipped with the (indefinite) Hermitian metric 
\[
\g{v}{v} = v_1\bar v_1 + v_2\bar v_2 - v_3\bar v_3.
\]
Let $\C^{2,1}_- = \{v\in\C^{2,1}:\g{v}{v}<0\}$, so that $\CH^2\simeq \P\C^{2,1}_-$. Thus we consider $\CH^2$ as the
orbit of the line 
$[0,0,1]\in\P \C^{2,1}_-$ under the standard action of
$G=PU(2,1)$. Consequently $\CH^2\simeq G/H$, where $H\simeq P(U(2)\times U(1))$ is a maximal compact subgroup of $G$. 
We equip $\CH^2$ with its Hermitian metric of constant holomorphic sectional curvature $-4$; so that its sectional
curvature has bounds $-4\leq \kappa\leq -1$. We write the Hermitian metric on $\CH^2$ as $h=g-i\omega$, where $\omega(X,Y) =
g(JX,Y)$, and recall that
$(\CH^2,h)$ is a K\"{a}hler-Einstein manifold. 

We will always think of a minimal immersion $f:\caD\to\CH^2$ as a conformal harmonic immersion, so the induced metric
$\gamma=f^*g$ is conformally equivalent to the hyperbolic metric $\mu$, with
$\gamma = e^u\mu$ for a smooth function $u:\Sigma\to\R$. To say that
$f$ is $\rho$-equivariant means it intertwines $\rho$ with a Fuchsian representation $\pi_1\Sigma\to \Isom(\caD)$.
The conjugacy class of
such a representation is equivalent to a choice of a \emph{marked conformal structure} on $\Sigma$, i.e., a point
$c\in\caT_g$ in the Teichm\"{u}ller space of $\Sigma$. We will write $\Sigma_c$ to denote the surface with this
structure. From now on we will assume $f$ is $\rho$-equivariant.

To understand the properties of such minimal immersions we need some notation for the type decomposition of the
(complexified) differential $df:T^\C\caD\to T^\C\CH^2$. 
Both the domain and the codomain are complex manifolds, so to 
distinguish between the type decompositions of their tangent vectors we will write
\[
T^\C\caD = T^{1,0}\caD\oplus T^{0,1}\caD,\quad T^\C\CH^2 = T'\CH^2+T''\CH^2.
\]
The projections will be such that $X = X^{1,0}+\overline{X^{1,0}}$ (or $X'+\overline{X'}$ as appropriate) whenever $X$ is real. 
Our primary model for $T^\C\CH^2$ will be the projective model: viz,
at any point $\ell\in\P\C^{2,1}_-$ we use the isomorphism
\begin{align}\label{eq:TCH}
T_\ell'\CH^2\oplus T_\ell''\CH^2&\to \Hom(\ell,\ell^\perp)\oplus\Hom(\ell^\perp,\ell)\subset \End(\C^{2,1});\\
(Z,W)&  \mapsto (\pi_\ell^\perp\circ Z,\pi_\ell\circ W),\notag
\end{align}
where $\pi_\ell:\C^{2,1}\to\ell$ is the orthogonal projection and we think of $Z,W$ as operations of differentiation on
local sections of $\CH^2\times\C^3$. In particular, conjugation in $T^\C\CH^2$ corresponds to taking the Hermitian transpose
in $\End(\C^{2,1})$, i.e., whose fixed subspace is $\fu(2,1)$. 
The isomorphism can be derived from the symmetric space model for $\CH^2$, which we will occasionally need to use (cf.\ the 
related model for $\CP^n$ in, for example, \cite{Bur95}). 
For that model, let $\fg= \su(2,1)$ and let $\fh\subset\fg$ denote
the Lie subalgebra for $H$. Then the symmetric space decomposition is $\fg = \fh+\fm$ where
$\fm=\fh^\perp$ with respect to the Killing form, and $T\CH^2\simeq [\fm]= G\times_H\fm$, where the action of $H$
on $\fm$ is its right adjoint action. It is easy to check that the fibre of $[\fm^\C]$ at $\ell$ agrees with the
codomain of \eqref{eq:TCH}, and that the metric corresponds to $g(A,B)=\tfrac12 \tr(AB)$ whenever $A,B\in\fm$.

By extending $\ell$ to mean the tautological sub-bundle the Hermitian metric $h$ on
$T'\CH^2$ is then equivalent to the inner product
\[
h(Z_1,Z_2) = \g{\pi_\ell^\perp Z_1\sigma_0}{\pi_\ell^\perp Z_2\sigma_0},\quad \sigma_0\in\Gamma(\ell),\ \g{\sigma_0}{\sigma_0}=-1.
\] 
Then type decomposition induces an isometry $T\CH^2\to T'\CH^2$.
These type decompositions give four complex linear parts of $df$:
\begin{align}\label{eq:df}
&\partial f':T^{1,0}\caD\to T'\CH^2,\quad \partial f'':T^{1,0}\caD\to T''\CH^2, \\
&\bar\partial f':T^{0,1}\caD\to T'\CH^2,\quad \bar\partial f'':T^{0,1}\caD\to T''\CH^2,
\end{align}
which are related by $\bar\partial f'' = \overline{\partial f'}$ and $\bar\partial f' = \overline{\partial f''}$ using
simultaneously the conjugation in $T^\C\caD$ and $T^\C\CH^2$.
Since $f$ is $\rho$-equivariant, so is $df$, i.e.,
$df d\delta = d\rho(\delta )df$ whenever $\delta \in\pi_1\Sigma\subset\Isom(\caD)$. Therefore we can think of $df$ as a
section of the bundle $T^\C\Sigma^*\otimes (f^{-1}T^\C\CH^2/\rho)$ over $\Sigma_c$. In particular,
\[
\partial f=\partial f'+\partial f'',
\]
is a smooth section of the vector bundle $K\otimes (f^{-1}T^\C\CH^2/\rho)$ over $\Sigma_c$,
where $K$ is the canonical bundle of $\Sigma_c$. 

One says that $f$ has a \emph{complex point} at $p=f(z)$ when $\partial f''$ vanishes at $p$ (i.e., 
when $df(T^{1,0}\caD)\subset
T'\CH^2$), and an \emph{anti-complex point} when $\partial f'$ vanishes at $p$.  
The Levi-Civita connexion induces a holomorphic structure on $f^{-1}T^\C\CH^2/\rho$ which preserves type decomposition,
and $f$ is harmonic when $\nabla^{\CH^2}_{\bar Z}\partial f(Z)=0$ for local holomorphic sections $Z$ of $T^{1,0}\caD$, i.e., 
when 
\[
\nabla^{\CH^2}_{\bar Z}\partial f'(Z)=0 \text{ and } \nabla^{\CH^2}_{\bar Z}\partial f''(Z)=0.
\]
Thus $\partial f'$ and $\partial f''$ are
holomorphic sections of their respective bundles; so a harmonic immersion which is not holomorphic or 
anti-holomorphic must have isolated anti-complex and complex points. We will denote the divisors of zeroes of $\partial
f'$ and $\partial f''$ on $\Sigma$ by $D_1$ and $D_2$ respectively .

Through these identifications there is a sesqui-linear form $h(df',df')$ on $T^\C\Sigma$ which gives
the induced metric as $\gamma = \Re h(df',df')=g(df',df')$. The map $f$ is \emph{weakly conformal} when
\[
h(\partial f',\bar\partial f')=0,
\]
and conformal when additionally $df'$ does not vanish. 
Therefore, for conformal maps the induced metric and the pull-back of the \Kah form are expressed,
with respect to a local complex coordinate $z$, as
\begin{equation}\label{eq:gamma}
\gamma= f^*g =(u_1^2+u_2^2)|dz|^2,\quad
f^*\omega = \frac{i}{2}(u_1^2-u_2^2)dz\wedge d\bar z.
\end{equation}
where
\begin{equation}\label{eq:u}
u_1 =\|\partial f'(Z)\|,\ u_2=\|\partial f''(Z)\|,\quad Z=\partial/\partial z.
\end{equation}
The functions $u_1,u_2$ are locally real analytic and vanish precisely at the anti-complex and complex points, respectively.
They correspond to Hermitian metrics
\begin{equation}\label{eq:gammaj}
\gamma_1 = h(\partial f',\partial f')=u_1^2dz\overline{dz},\quad \gamma_2=h(\partial f'',\partial f'')= u_2^2dz\overline{dz}
\end{equation}
on $K^{-1}(D_1)$ and $K^{-1}(D_2)$ respectively. Note that, to apply $\gamma$ to all elements of
$T^\C\Sigma$ consistently,  the meaning of ``$|dz|^2$'' above is
\[
|dz|^2 = \tfrac12 (dz\overline{dz} + d\bar z\overline{d\bar z}),
\]
in terms of the local complex linear forms $dz,d\bar z$. For this reason we do not write $\gamma=\gamma_1+\gamma_2$.

Because the forms $\gamma,f^*\omega$ live on $\Sigma$
we can use some of the arguments which apply to compact minimal surfaces in
K\"{a}hler-Einstein manifolds \cite{Web,Wol89,CheT} to relate numerical invariants of a minimal immersion.
\begin{thm} Let $f:\caD\to\CH^2$ be a $\rho$-equivariant minimal immersion which is neither holomorphic nor
anti-holomorphic. Let $d_1,d_2$ be the degrees of the divisors $D_1$ and $D_2$ of anti-complex and complex points. 
Then
\begin{eqnarray}
c_1(\rho)& =& d_1-d_2,\label{eq:c_1}\\
\chi(\Sigma) + \chi(T\Sigma^\perp)& =& -d_1-d_2, \label{eq:chi}
\end{eqnarray}
where $c_1(\rho)$ is the first Chern class of the bundle $f^{-1}T\CH^2/\rho$ over $\Sigma$ and $T\Sigma^\perp\subset
f^{-1}T\CH^2/\rho$ is the normal sub-bundle. 
\end{thm}
Wolfson also showed that in the absence of complex or anti-complex points a minimal immersion into a
K\"{a}hler-Einstein surface of negative scalar curvature must be Lagrangian \cite[Thm 2.1]{Wol89}. His argument 
extends to $\rho$-equivariant maps, so that in the setting of the previous theorem $f$ will be 
Lagrangian if and only if $d_1=0=d_2$.

Now we recall that the Toledo invariant $\tau(\rho)$ is defined by
\begin{equation}\label{eq:taudef}
\tau(\rho) = \frac{2}{\pi}\int_\Sigma f^*\omega.
\end{equation}
This is the normalisation which fits with the metric of  holomorphic sectional curvature $-4$.
It is known that for any representation $\rho$ into $PU(2,1)$, $|\tau(\rho)|\leq -\chi(\Sigma)$
\cite{DomT} and $\tau(\rho)\in \tfrac{2}{3}\Z$ \cite{GolKL}.
Since $\CH^2$ has Einstein constant $-6$, equation \eqref{eq:c_1} tells us that
\begin{equation}\label{eq:tau}
\tau(\rho) =  -\tfrac{2}{3}c_1(\rho)= \tfrac{2}{3} (d_2-d_1).
\end{equation}
Finally, as well as the degenerate metrics $\gamma_j$ above, there is a third important invariant of minimal equivariant 
immersions \cite[Cor 2.7]{Woo84}, 
the cubic holomorphic differential $\caQ\in H^0(\Sigma,K^3)$ defined by
\begin{equation}\label{eq:Q}
\caQ(Z,Z,Z) = h(\nabla^{\CH^2}_Z\partial f'(Z),\bar\partial f'(\bar Z))
=-h(\partial f'(Z),\nabla^{\CH^2}_{\bar Z}\bar\partial f'(\bar Z)),
\end{equation}
for $Z\in T^{1,0}\caD$.
It follows at once from this that $\caQ$ vanishes identically for holomorphic or anti-holomorphic immersions (but not only
for them, as we shall see). When $f$ is neither holomorphic nor anti-holomorphic
we will see later that the quantities $\gamma_1,\gamma_2,\caQ$ uniquely determine $f$ up to ambient isometries.

\subsection{$G$-Higgs bundles and representations.}
As well as fixing our notation for Higgs bundles, we also need to summarise their correspondence with (projectively)
flat connexions, and hence representations and harmonic maps, since we will be making explicit use of this correspondence 
for most of this article.  A good general introduction to Higgs bundles can be found in, for example, \cite{Garapp,Gar}, 
while details for their moduli spaces in the case $G=PU(2,1)$ can be found in \cite{Got,Xia}.

Suppose $\Sigma$ has been given a fixed conformal structure. With the notation of the previous section, a $G$-Higgs 
bundle for $G=PU(2,1)$ is a projective equivalence class of $U(2,1)$-Higgs bundles. The latter are
pairs $(E,\Phi)$ consisting of a holomorphic rank three vector bundle $E$ over $\Sigma$ equipped
with a splitting $E=V\oplus L$ into a rank two sub-bundle $V$ and a line sub-bundle $L$ (both holomorphic) together with a 
holomorphic section
\begin{equation}\label{eq:Higgs}
\Phi\in H^0(K\otimes[\Hom(L,V) \oplus \Hom(V,L)]),
\end{equation}
called the \emph{Higgs field}. Projective equivalence identifies pairs $(E,\Phi)$ and $(E',\Phi')$ when 
there is a holomorphic line bundle $\caL$ for which $E'=E\otimes\caL$ and $\Phi'=\Phi$. 

We will write $\Phi = (\Phi_1,\Phi_2)$ to denote the two summands implied by the direct sum \eqref{eq:Higgs}.
It is also convenient to write the holomorphic structure on $E$ as a
$\bar\partial$-operator on smooth sections, $\bar\partial_E:\caE^0(E)\to \caE^{0,1}(E)$.
A Higgs bundle is \emph{stable} if for any proper (non-zero) $\Phi$-invariant sub-bundle $W\subset E$ the slope condition
\begin{equation}\label{eq:slope}
\frac{\deg(W)}{\rk(W)} < \tfrac13\deg(E),
\end{equation}
is satisfied. It is \emph{polystable} when it is either stable or the direct sum of stable proper Higgs sub-bundles all 
having the same slope (these latter type are called \emph{strictly polystable}). These properties are independent of the
choice of pair $(E,\Phi)$ representing the projective equivalence class. Now we recall that each polystable $U(2,1)$ Higgs bundle
admits a $\C^{2,1}$ metric and compatible projectively flat connexion, and thereby produces a flat $PU(2,1)$-bundle. 
\begin{thm}[Prop 3.9, \cite{BraGG}]
Fix a K\"{a}hler metric on $\Sigma_c$ in the conformal class $c$, with  K\"{a}hler $2$-form $\omega_c$, normalised so that
$\Sigma_c$ has area $2\pi$.
For each polystable stable $U(2,1)$-Higgs bundle $(E,\Phi)$ of slope $\mu$ there is a $\C^{2,1}$ metric on $E$ for 
which $L=V^\perp$ and
the corresponding Chern connexion $\nabla_E$ and conjugate $\Phi^\dagger$ yield a projectively flat connexion 
$\nabla = \nabla_E + \Phi+\Phi^\dagger$, i.e., 
\begin{equation}\label{eq:projflat}
R^\nabla = R^{\nabla_E}+[\Phi\wedge\Phi^\dagger] = -i\mu \omega_c I_E.
\end{equation}
\end{thm}
Here $\Phi^\dagger$ denotes the conjugate with respect to the $\fu(2,1)$ real form, so
the connexion $\nabla$ induces a flat connexion on the principal $PU(2,1)$-bundle whose associated bundle is $\P E$.
The holonomy of this flat connexion yields a reductive representation
$\rho:\pi_1\Sigma\to PU(2,1)$ (determined up to conjugacy) which is irreducible precisely when $(E,\Phi)$ is stable
(hence strictly polystable Higgs bundles correspond to reducible reductive representations).
\begin{rem}
This theorem can be thought of as an example of the Donaldson-Uhlenbeck-Yau correspondence between stable bundles over compact
\Kah manifolds and Hermitian-Einstein connections (Garcia-Prada gives a good overview which fits our context in \cite{Garapp}). 
For $G$-Higgs bundles such a correspondence is due to Hitchin for $G=SL(2,\C)$ \cite{Hit87} and Simpson \cite{Sim} when 
$G$ is a complex reductive algebraic group. 
\end{rem}
Note that the usual statement of the previous theorem gives the existence of a Hermitian metric on $E$. This is
equivalent, since one can simply swap the sign of the metric on $L$, and the condition \eqref{eq:Higgs}
on the Higgs field ensures that $\Phi^\dagger$ corresponds to the adjoint for the Hermitian metric.
This makes the $\C^{2,1}$ metric negative definite on $L$, and therefore $L$ determines a
smooth section of the $\CH^2$ bundle $\P E_-$ (where $E_-$ denotes the bundle of negative length vectors in $E$).
Since $\P E_-\simeq \caD\times_\rho\CH^2$ this section is equivalent to a $\rho$-equivariant
map $f:\caD\to\CH^2$ in such a way that $\Phi = \partial f$. Moreover, $\nabla_E$ induces a metric connexion on
$f^{-1}T\CH^2$ which agrees with the pull-back of the Levi-Civita connexion, so that the equation $\bar\partial_E\Phi=0$ 
is the harmonic map condition for $f$.

In the reverse direction, a representation $\rho:\pi_1\Sigma\to G$ determines the projective bundle
$\P E$ uniquely and therefore a
class of projectively equivalent $\C^{2,1}$ bundles, each with a projectively flat connexion. By Corlette's results
\cite{Cor}, there is a $\rho$-equivariant harmonic map $f:\caD\to\CH^2$ precisely when $\rho$ is reductive, and this map is
unique when $\rho$ is irreducible . The map corresponds to a line sub-bundle $L\subset E$ and therefore a splitting 
$E=L^\perp\oplus L$.
The splitting determines a bundle automorphism $\sigma\in\End(E,E)$ for which $\sigma|L^\perp=1$ and $\sigma|L=-1$, and
therefore a decomposition $\nabla = \nabla_E+\Psi$, where
\[
\nabla_E = \tfrac12(\nabla + \sigma\nabla\sigma),\quad \Psi = \tfrac12(\nabla - \sigma\nabla\sigma).
\]
The Higgs field is $\Phi=\Psi^{1,0}$.
The harmonic map equation, when paired with the projective flatness of $\nabla$, asserts that
$\bar\partial_E\Psi^{1,0}=0$, and thus the Higgs field satisfies \eqref{eq:Higgs} when we take
$V=L^\perp$.

Two such bundles, $(E,\nabla)$ and $(E',\nabla')$, are projectively equivalent when there is a line bundle $\caL$ 
equipped with a connexion $\nabla_\caL$ for which $E'\simeq E\otimes\caL$ and $\nabla'\simeq\nabla\otimes\nabla_\caL$
(the induced connexion on the tensor product). In particular, by taking $\caL = L^{-1}$ equipped with the connexion 
obtained from the restriction of $\nabla$ to $L$, we may assume without loss of generality that $E=V\oplus 1$, where 
$1$ denotes the trivial bundle, and the restriction of $\nabla$ to $1$  is the canonical flat connexion. In that case the Toledo 
invariant of $\rho$ is $-\tfrac23\deg(V)$.
\begin{rem}
The alternative normalisation, used by Xia \cite{Xia}, is to note that since $\deg(E\otimes\caL) =
\deg(E)+3\deg(\caL)$ one can normalise by degree, i.e., insist that $0\leq\deg(E)<3$. In particular, the topological type
of $\P E$ is determined by $\deg(E)\bmod 3$, and the representation $\rho$ only lifts to $SU(2,1)$ when there exists an
$\caL$ for which $E\otimes\caL\simeq\caD\times_{\hat\rho}\C^{2,1}$ for some representation $\hat\rho:\pi_1\Sigma\to U(2,1)$. 
This happens if and only if $\deg(E)\equiv 0\bmod 3$, i.e., when $\tau\in 2\Z$.
\end{rem}
\begin{rem}
From the Higgs bundle perspective, the Toledo invariant is sometimes defined to be $\tfrac23\deg(V L^{-1})$ (see, for
example, \cite{Xia}). This differs by a sign from our convention, since
\begin{equation}\label{eq:Toledo}
\tfrac23 \deg(V L^{-1}) = \tfrac23 \deg\Hom(L,V) = \tfrac23 c_1(\rho) =-\tau(\rho).
\end{equation}
\end{rem}

\subsection{Minimal surfaces and their Higgs bundles.}
We are now in a position to classify, in terms of Higgs bundle data, the minimal $\rho$-equivariant surfaces which 
are neither holomorphic nor anti-holomorphic, when $\rho$ is irreducible.
\begin{thm}\label{thm:minimal}
An irreducible representation 
$\rho\in \Hom(\pi_1\Sigma,G)$ admits a $\rho$-equivariant minimal immersion 
$f:\caD\to\CH^2$ which is neither holomorphic nor anti-holomorphic if and only if it 
corresponds to a Higgs bundle $(E,\Phi)$ for which $E=V\oplus 1$ where 
$V$ is a rank $2$ holomorphic extension 
\begin{equation}\label{eq:V}
0\to K^{-1}(D_1)\stackrel{\Phi_1}{\to} V\stackrel{\Phi_2}{\to} K(-D_2)\to 0.
\end{equation}
Here $D_1,D_2$ are effective divisors with no common points, whose degrees $d_1,d_2$ 
satisfy the stability inequalities
\begin{equation}\label{eq:stab}
2d_1+d_2<6(g-1) \text{ and } d_1+2d_2<6(g-1),
\end{equation}
where $g$ is the genus of $\Sigma$. The Higgs field is given by $\Phi = (\Phi_1,\Phi_2)$,
after making the canonical identifications
\begin{align}\label{eq:ident}
&\Hom(K^{-1}(D_1),V)\simeq K(-D_1)\otimes\Hom(1,V), \notag \\
&\Hom(V, K(-D_2))\simeq K(-D_2)\otimes\Hom(V,1).
\end{align}
These divisors are, respectively, the divisors of anti-complex and complex points of the minimal immersion $f$.
The representation $\rho$ has $\tau(\rho) = \tfrac23(d_2-d_1)$.
\end{thm}
\begin{proof}
Given $\rho$ we obtain a projectively flat $\C^{2,1}$ bundle of the form $E=V\oplus 1$, and then $f$ provides 
holomorphic sections
\begin{align*}
\Phi_1&=\partial f'\in H^0(K\otimes\Hom(1,V))\simeq H^0(\Hom(K^{-1},V)),\\ 
\Phi_2&=\partial f'' \in H^0(K\otimes\Hom(V,1))\simeq H^0(\Hom(V,K)),
\end{align*}
whose sum is $\Phi=\partial f$. Since $f$ is conformal we have 
\[
0=g(\partial f,\partial f)=\tfrac12 \tr(\Phi^2).
\]
From this we can show that $\Phi_2\circ\Phi_1=0$, by the following local frame argument. 
Neither of $\Phi_1,\Phi_2$ are identically zero since
$f$ is not $\pm$-holomorphic. So, away from its zero locus, the image of $\Phi_1$ is a rank one sub-bundle $V_1\subset V$.
We can locally frame $E$ by sections $\sigma_1,\sigma_2$ of $V$ and $\sigma_3$ of $1$ such that $\sigma_1$ generates
$V_1$,
and we may assume this is a $U(2,1)$ frame with respect to the metric on $E$. It follows that there are
locally holomorphic sections $a,b,c$ of $K$ for which
\[
\Phi_2(\sigma_1) = a\sigma_3,\quad \Phi_2(\sigma_2) = b\sigma_3,\quad \Phi_1(\sigma_3) = c\sigma_1.
\]
Since $f$ is not anti-holomorphic $c\neq 0$. Now 
\begin{align*}
\tr( \Phi^2) & = \tr(\Phi_1\circ\Phi_2+ \Phi_2\circ\Phi_1)\\
&= \g{\Phi_1\circ\Phi_2(\sigma_1)}{\sigma_1} +\g{\Phi_1\circ\Phi_2(\sigma_2)}{\sigma_2} -
\g{\Phi_2\circ\Phi_1(\sigma_3)}{\sigma_3}\\
& = 2ac. 
\end{align*}
Therefore $\tr(\Phi^2)=0$ implies $a=0$, i.e.,  $\Phi_2\circ\Phi_1=0$.
Thus 
\[
K^{-1}\stackrel{\Phi_1}{\to}V\stackrel{\Phi_2}{\to} K,
\]
has the image of $\Phi_1$ in the kernel of $\Phi_2$. Now $\Phi_1$ vanishes precisely 
on anti-complex points, while $\Phi_2$ vanishes precisely on complex points, so we have
\[
0\to K^{-1}(D_1)\stackrel{\Phi_1}{\to} V\stackrel{\Phi_2}{\to} K(-D_2)\to 0.
\]
This must be exact at the middle since $V$ has rank $2$.

For stability we need to identify the $\Phi$-invariant sub-bundles. With respect to the local frame 
$\sigma_1,\sigma_2,\sigma_3$ above $\Phi$ is represented by the matrix
\[
\begin{pmatrix} 0&0&c\\ 0&0&0\\ 0&b&0\end{pmatrix}.
\]
It follows that the only $\Phi$-invariant proper sub-bundles of $E$ are the image $V_1$ of $\Phi_1$ and $V_1\oplus 1$. 
So stability requires
\[
\deg(V_1)<\tfrac13\deg(E),\quad \tfrac12\deg(V_1)<\tfrac13\deg(E).
\]
Since $\deg(V_1) = d_1-2(g-1)$ and $\deg(E) = d_1-d_2$ these inequalities are equivalent to the inequalities
\eqref{eq:stab}.

Reversing the argument is straightforward: when the Higgs bundle has this form we have $\Phi_2\circ\Phi_1=0$, hence
$\tr(\Phi^2)=0$. So when the Higgs bundle is stable the sub-bundle $1\subset E$ provides a conformal harmonic 
$\rho$-equivariant map $f$ for some irreducible $\rho$.
\end{proof}
The proof works perfectly well in the case where
$D_1,D_2$ have common points, in which case the map $f$ is a branched minimal immersion with branch points on $D_1\cap
D_2$. We will follow common usage and say $D_1,D_2$ are \emph{co-prime} when they are disjoint. 

The Higgs bundle data which appears in the previous theorem can be written as a quadruple
$(\Sigma_c,D_1,D_2,\xi)$, where $\xi\in
H^1(\Sigma_c,K^{-2}(D_1+D_2))$ is the extension class of \eqref{eq:V}.  The theorem shows that this
data determines the pair $(\rho,f)$ up to $G$-equivalance, i.e., up to the simultaneous action of $G$
by conjugacy of $\rho$ and ambient isometry of $f$. The exact sequence \eqref{eq:V}, which
gives us $V$ and $\Phi_1,\Phi_2$, is completely and uniquely determined by its extension class $\xi$ \cite[Ch. 5]{Gun}.
Moreover, since the isomorphisms \eqref{eq:ident} require $D_1,D_2$, not just their linear equivalence classes,
the assignment from $(\Sigma_c,D_1,D_2,\xi)$ to the $G$-equivalance class of $(\rho,f)$ is bijective. 
The moduli space parametrised by this data will be described in \S 6.
\begin{rem}\label{rem:hol}
The proof above shows that when $(E,\Phi)$ has $\tr(\Phi^2)=0$ it cannot split into a direct sum of
$\Phi$-invariant sub-bundles unless one of $\Phi_1$ or $\Phi_2$ is identically zero. It follows that
a  reducible $\rho$ can  only admit $\rho$-equivariant minimal surfaces which are holomorphic or anti-holomorphic. 
We describe all these reducible representations in Remark \ref{rem:reducible} below.
However, there are also irreducible representations which admit
holomorphic or anti-holomorphic surfaces: we give a complete classification in \S\ref{sec:Q=0}. 
\end{rem}

\section{The Higgs bundle data in terms of minimal surface data.}

In this section we will give the explicit correspondence between the minimal surface data $(\gamma_1,\gamma_2,\caQ)$ 
and the Higgs bundle data $(\Sigma_c,D_1,D_2,\xi)$ of the previous section. 
This is achieved by exploiting the harmonic sequence for a minimal surface in $\CH^2$.
It provides us with a preferred system of local $U(2,1)$ frames for the bundle $E$ in Theorem \ref{thm:minimal}, and which we 
will call \emph{Toda frames}. These frames are explictly determined by the minimal surface data, and through them we calculate 
the extension class $\xi$ of the bundle $V$ in Theorem \ref{thm:minimal}. 
The correspondence between the Higgs data and the minimal surface data then comes through the Dolbeault isomorphism
\[
H^1(\Sigma_c,K^{-2}(D_1+D_2))\simeq H^{0,1}(\Sigma_c,K^{-2}(D_1+D_2)).
\]
\begin{thm}\label{thm:xi}
Let the pair $(\rho,f)$ correspond to the Higgs data $(\Sigma_c,D_1,D_2,\xi)$ as in the previous section. 
Let  $\gamma_1,\gamma_2,\caQ$ be the minimal surface data determined by $f$ through \eqref{eq:gammaj} and \eqref{eq:Q}.
Then the extension class $\xi$ corresponds, under the Dolbeault isomorphism, to the cohomology class
\[
-\left[\frac{\bar\caQ}{\gamma_1\gamma_2}\right]\in H^{0,1}(\Sigma_c,K^{-2}(D_1+D_2)).
\]
Moreover $\xi=0$ if and only if $\caQ=0$.
\end{thm}
In particular, this means that $(\gamma_1,\gamma_2,\caQ)$ determines $(\Sigma_c,D_1,D_2,\xi)$, since we get the conformal
structure of the induced metric and the divisors $D_1,D_2$ from $\gamma_1,\gamma_2$. Therefore Theorems 
\ref{thm:minimal} and \ref{thm:xi} together have the following corollary. 
\begin{cor}
If two $\rho$-equivariant minimal surfaces have the same data $(\gamma_1,\gamma_2,\caQ)$ then they are identical up to ambient
isometry.
\end{cor}
Before we begin the proof of Theorem \ref{thm:xi} we describe the local Toda frames which link the minimal surface data to the 
local geometry of the Higgs bundle. Fix a pair $(\rho,f)$ and let $E=V\oplus 1$ be the Higgs bundle over
$\Sigma_c$ corresponding to them as above, equipped with its Higgs field $\Phi$, its $\C^{2,1}$ metric and the Chern 
connexion $\nabla_E$.  
\begin{lem}\label{lem:Todaframe}
Let $(U,z)$ be holomorphic chart on $\Sigma_c$ for which $U$ contains no complex or anti-complex points
of $f$. Then over $U$ there is a local trivialisation $\varphi:E|U\to U\times\C^3$ for which
\begin{equation}\label{eq:dbarE}
\varphi\circ\bar\partial_E\circ\varphi^{-1}= 
d\bar z\left[\frac{\partial}{\partial\bar z} +
\begin{pmatrix} -\bar Z\log u_1 & -\bar Q/u_1u_2 & 0\\ 0 & \bar Z\log u_2 & 0\\ 0 & 0 & 0
\end{pmatrix}\right]
\end{equation}
and 
\begin{equation}\label{eq:Phi}
\varphi\circ\Phi\circ\varphi^{-1} = 
\begin{pmatrix} 0 & 0 & u_1\\ 0 & 0 & 0\\ 0 & u_2 & 0 \end{pmatrix}dz,
\end{equation}
where $\bar Z=\partial/\partial\bar z$ and $u_1,u_2,Q$ are given by \eqref{eq:u} and \eqref{eq:Q}.
\end{lem}
\begin{proof}
For notational convenience, let us set $\ell_0=1\subset E$. Then $V=\ell_0^\perp$
with respect to the $\C^{2,1}$ metric, and $\ell_0$ is the section of $\P E_-$ which represents $f$. Thus $\partial
f',\bar\partial f'\in\Omega_\Sigma^1(\Hom(\ell_0,\ell_0^\perp))$ and they determine line sub-bundles
$\ell_1,\ell_2\subset\ell_0^\perp$ via their images, i.e.,
\begin{equation}\label{eq:harmseq}
\ell_2\otimes\bar K\stackrel{\bar\partial f'}{\leftarrow}\ell_0\stackrel{\partial f'}{\rightarrow}\ell_1\otimes K.
\end{equation}
These two sub-bundles are orthogonal since $f$ is conformal. 
We give each of these line bundles the holomorphic structure it inherits from $\bar\partial_E$, i.e., a local section
$\sigma_j$ of $\ell_j$ is holomorphic when $\pi_j\bar\partial_E\sigma_j=0$, where $\pi_j:E\to\ell_j$ is the orthogonal
projection. Note that since $\Phi\in\Hom(\ell_0,\ell_0^\perp)$ this holomorphic structure on each $\ell_j$ agrees with the
one induced by the projectively flat connexion $\nabla=\nabla_E+\Phi+\Phi^\dagger$. Then $\partial f'=\Phi_1$ is a holomorphic 
map, while $\bar\partial f'=\Phi_2^\dagger$ is anti-holomorphic, since $f$ is harmonic. 
Since $\nabla$ induces the canonical flat connexion on $1$ we can choose a globally flat section $f_0$ of $1$, i.e., $\g{\nabla
f_0}{f_0}=0$ with $\g{f_0}{f_0}=-1$.

Now fix a chart $(U,z)$ and for any section $\sigma$ of $E|U$ write $X\sigma$ to mean $\nabla_X\sigma$ with respect to the
projectively flat connexion $\nabla$. Then the maps above can be written locally as
\[
\ell_2\stackrel{\pi_0^\perp\bar Z}{\leftarrow}\ell_0\stackrel{\pi_0^\perp Z}{\rightarrow}\ell_1.
\]
Define local sections $\sigma_j\in\Gamma(U,\ell_j)$, $j=1,2$ by 
\[
\sigma_1 = \partial f'(Z)f_0 = \pi_0^\perp Zf_0,\quad \sigma_2 = \bar\partial f'(\bar Z)f_0 = \pi_0^\perp\bar Z f_0.
\]
We assume that $U$ does not contain any complex or anti-complex points, and therefore the functions $u_1,u_2$ in
\eqref{eq:u} are non-vanishing.
Clearly $|\sigma_j|=u_j$, and we set $f_j = \sigma_j/u_j$. Then $f_1,f_2,f_0$ is a $U(2,1)$
frame for $E$. We claim that these satisfy the equations
\begin{align}
Zf_1 & =  (Z\log u_1)f_1 + (Q/u_1u_2)f_2,\label{eq:frame1}\\
Zf_2 & =  -(Z\log u_2)f_2 + u_2f_0,\label{eq:frame2}\\
Zf_0 & =  u_1f_1.\label{eq:frame0}
\end{align}
The last of these is obvious, since $Zf_0 = \pi_0^\perp Zf_0$. 

Next consider
\[
Zf_1 = \g{Zf_1}{f_1}f_1+\g{Zf_1}{f_2}f_2.
\]
Since $f_0$ is holomorphic, so is $\sigma_1$, and therefore $Z\g{\sigma_1}{\sigma_1} = \g{Z\sigma_1}{\sigma_1}$. Thus
\begin{align*}
\g{Zf_1}{f_1} & = u_1^{-1}\g{Z(u_1^{-1}\sigma_1)}{\sigma_1}\\
& = -Z\log(u_1) + u_1^{-2}Z\g{\sigma_1}{\sigma_1}\\
& = Z\log(u_1).
\end{align*}
Now
\begin{align*}
\g{Zf_1}{f_2} & = \g{Z(\sigma_1/u_1)}{\sigma_2/u_2}\\
& = \g{Z\sigma_1}{\sigma_2}/u_1u_2\\
& = \g{\pi_0^\perp Z\pi_0^\perp Zf_0}{\pi_0^\perp\bar Zf_0}/u_1u_2.
\end{align*}
On the other hand, using \eqref{eq:Q} and the fact that $\nabla^{\CH^2}$ is the connexion on
$\Hom(\ell_0,\ell_0^\perp)$ induced by the connexions on each bundle $\ell_0,\ell_0^\perp$, 
we have
\begin{eqnarray}
Q& =& \g{[\nabla^{\CH^2}_Z\partial f'(Z)]f_0}{\bar\partial f'(\bar Z)) f_0} \notag \\
& =& \g{\pi_0^\perp Z[\partial f'(Z)f_0] -\partial f'(Z)[\pi_0Zf_0] }{\bar\partial f'(\bar Z)) f_0} \notag \\
&=& \g{\pi_0^\perp Z\pi_0^\perp Zf_0}{\pi_0^\perp\bar Z f_0}
-\g{\pi_0^\perp Z\pi_0 Zf_0}{\pi_0^\perp\bar Z f_0}.\label{eq:Qf0}
\end{eqnarray}
The second term in the last line vanishes since $\ell_{-1}$ is orthogonal to $\ell_1$. Thus
$\g{Zf_1}{f_2} = Q/u_1u_2$.

Finally, consider
\[
Zf_2 = \g{Zf_2}{f_2}f_2 - \g{Zf_2}{f_0}f_0.
\]
For the first term we note that $\sigma_2$ is anti-holomorphic since $f_0$ is, so
\[
0=\g{Z\sigma_2}{f_2} = Z u_2 +u_2\g{Zf_2}{f_2}.
\]
Since $\g{f_2}{f_0}=0$ the second term yields
\[
- \g{Zf_2}{f_0} = \g{f_2}{\bar Z f_0} = \g{f_2}{\sigma_2} = u_2.
\]
Now that we have established the equations for the frame $f_1,f_2,f_0$, we take $\varphi$ to be the 
corresponding trivialisation. In this frame the equations \eqref{eq:frame1}-\eqref{eq:frame0} show us that
\begin{equation}\label{eq:nabla}
\varphi\circ\nabla^{1,0}\circ\varphi^{-1} =
dz\left(\frac{\partial}{\partial z} +
\begin{pmatrix} Z\log u_1 & 0 & u_1\\ Q/u_1u_2 &  -Z\log u_2 & 0\\ 0 & u_2 & 0
\end{pmatrix}\right)
\end{equation}
Thus \eqref{eq:dbarE} and \eqref{eq:Phi} follow.
\end{proof}
At complex or anti-complex points we require a slight adjustment of the frame above. We may choose the chart $(U,z)$ so
that $U$ contains precisely one complex or anti-complex point, at $z=0$. Thus one of $\sigma_1,\sigma_2$ vanishes at $z=0$.
To treat all cases simultaneously, let $p,q$ be the non-negative integers for which
$z^{-q}\sigma_1$ and $\bar z^{-p}\sigma_2$ are respectively holomorphic and anti-holomorphic, and do not vanish at 
$z=0$: at most one of $p,q$ is non-zero.  It follows that the
functions $u_1/|z|^q = |z^{-q}\sigma_1|$ and $u_2/|z|^p = |\bar z^{-p}\sigma_2|$ do not vanish. By setting
\[
\tilde f_1 = \frac{z^{-q}\sigma_1}{|z^{-q}\sigma_1|} = (|z|/z)^qf_1,\quad
\tilde f_2 = (|z|/\bar z)^pf_2= (z/|z|)^pf_2,
\]
we obtain a $U(2)$ frame $\tilde f_1,\tilde f_2$ for $V$ throughout $U$, with corresponding trivialisation $\tilde\varphi$. 
This is easiest to work with by writing it as
\begin{equation}\label{eq:S}
\tilde\varphi = S(\varphi|V),\quad S = \begin{pmatrix} (z/|z|)^q &0\\
0&(|z|/z)^p \end{pmatrix}.
\end{equation}
The next step towards Theorem \ref{thm:xi} is to represent the extension class $\xi$ of $V$ using a $1$-cocycle with
values in $K^{-2}(D_1+D_2)$. For this computation, we choose an atlas $\caU=\{(U_j,z_j)\}$ for $\Sigma_c$ of contractible
charts for which each complex or anti-complex point lies only in one chart, at $z_j=0$,
and for simplicity assume charts containing distinct complex or anti-complex points are disjoint.
\begin{lem}\label{lem:xi}
There is a holomorphic atlas $\caU=\{(U_j,z_j)\}$ for $\Sigma_c$ in which $\xi$ is represented by the $1$-cocycle
$\{(\xi_{jk},U_j,U_k)\}\in H^1(\caU,K^{-2}(D_1+D_2))$ for which
\begin{equation}\label{eq:xijk}
\xi_{jk} = a_jz^{-(p_j+q_j)} dz_j^{-2} - a_kdz_k^{-2}\ \text{on}\ U_j\cap U_k,
\end{equation}
for smooth functions $a_j$ on $U_j$ satisfying 
\[
\partial a_j/\partial\bar z_j = -\frac{\bar Q_jz_j^{p_j+q_j}}{u_{1j}^2u_{2j}^2},
\]
(assuming $U_k$ contains no complex or anti-complex points).
\end{lem}
\begin{proof}
For simplicity, set $w = z/|z|$. Using the previous lemma and the transformation \eqref{eq:S} we have, on $V$ 
over a single chart $(U,z)$, 
\[
\tilde\varphi\circ\bar\partial_E\circ\tilde\varphi^{-1} = 
d\bar z\left[\frac{\partial}{\partial\bar z} +
\begin{pmatrix} -\bar Z\log (u_1/|z|^{q}) & -\bar Qw^{p+q}/u_1u_2 \\ 0 & \bar Z\log(u_2/|z|^{p}) 
\end{pmatrix} \right].
\]
Now suppose $U_j\cap U_k\neq \emptyset$, and assume
without loss of generality that $U_k$ contains no complex or anti-complex points, so that $p_k=0=q_k$ and
$\tilde\varphi_k = \varphi_k|V$. Then we have transition relations $\tilde\varphi_j=c_{jk}\tilde\varphi_k$ where
\[
c_{jk} = \begin{pmatrix} w_j^{q_j}\frac{dz_j/dz_k}{|dz_j/dz_k|} & 0 \\ 0 & w_j^{-p_j}\frac{|dz_j/dz_k|}{dz_j/dz_k} 
\end{pmatrix},
\]
where we have used that fact that
\[
\frac{d\bar z_j/d\bar z_k}{|d\bar z_j/d\bar z_k|} = \overline{\frac{dz_j/dz_k}{|dz_j/dz_k|}}= 
\frac{|dz_j/dz_k|}{dz_j/dz_k}.
\]
It follows that, as a smooth bundle, $V\simeq K^{-1}(D_1)\oplus K(-D_2)$.

To elucidate its holomorphic structure we find local trivialisations $\chi_j$ in which
\[
\chi_j\circ \bar\partial_E\circ\chi_j^{-1} = d\bar z_j\frac{\partial}{\partial\bar z_j},
\]
i.e., we seek local gauge transformations $\chi_j = R_j\tilde\varphi_j$ for which
\[
R_j[\tilde\varphi\circ\bar\partial_E\circ\tilde\varphi^{-1}]R_j^{-1} = d\bar z_j
 \frac{\partial}{\partial\bar z_j}.
\]
A straightforward calculation shows that this is achieved by taking
\[
R_j = \begin{pmatrix} 1&a_j\\0&1\end{pmatrix}
\begin{pmatrix} (u_{1j}/|z_j|^{q_j})^{-1}&0\\0&u_{2j}/|z_j|^{p_j}\end{pmatrix},
\]
where 
\begin{equation}\label{eq:aj}
\frac{\partial a_j}{\partial\bar z_j} = -\frac{\bar Q_jz^{p_j+q_j}}{u_{1j}^2u_{2j}^2},
\end{equation}
throughout $U_j$. Such a function $a_j$ exists, by the $\bar\partial$-Poincar\'{e} Lemma, because the right 
hand side is smooth throughout $U_j$ since $Q_j$ 
vanishes to at least order $p_j+q_j$ at $z_j=0$ by \eqref{eq:Q}. Thus the transition between two such 
charts $\chi_j=b_{jk}\chi_k$ is given by
\[
b_{jk} = R_jc_{jk}R_k^{-1} = \begin{pmatrix} z_j^{q_j}dz_j/dz_k & \lambda_{jk}\\ 0 & z_j^{-p_j}dz_k/dz_j\end{pmatrix},
\]
where
\[
\lambda_{jk} = a_jz_j^{-p_j}\frac{dz_k}{dz_j} - a_kz_j^{q_j}\frac{dz_j}{dz_k},
\]
and this is holomorphic on $U_j\cap U_k$ (this follows easily from \eqref{eq:aj}). Now using the same 
convention as \cite[p74]{Gun} this determines the $1$-cocycle with values in $K^{-2}(D_1+D_2)$ given by
\begin{align*}
\xi_{jk} & = (z_j^{q_j}\frac{dz_j}{dz_k})^{-1}\lambda_{jk}dz_k^{-2} \\
& = a_jz_j^{-(p_j+q_j)}dz_j^{-2} - a_kdz_k^{-2}.
\end{align*}
\end{proof}
\begin{proof}[Proof of Theorem \ref{thm:xi}]
For notational simplicity, set $\caL = K^{-2}(D_1+D_2)$. In $U_j$ the quantity 
\[
z_j^{-(p_j+q_j)}dz_j^{-2}
\]
is a local holomorphic section of $\caL$. Therefore we have local smooth sections
\[
\eta_j=a_jz_j^{-(p_j+q_j)}dz_j^{-2} \in \Gamma(U_j,\caE^0(\caL)),
\]
which provide a $0$-cochain $\eta$ for the sheaf $\caE^0(\caL)$ of locally smooth sections of $\caL$.
By \eqref{eq:xijk} the $1$-cocyle $\{(\xi_{jk},U_j,U_k)\}$ is, as a smooth cocycle, the coboundary of $\eta$. 
Now recall that the Dolbeault isomorphism $H^1(\caU,\caL)\to H^{0,1}(\caU, \caL)$
is derived from the short exact sequence
\[
0\to \caO(\caL)\to\caE^0(\caL)\stackrel{\bar\partial}{\to} \caE^{0,1}(\caL)\stackrel{\bar\partial}{\to} 0,
\]
by taking the cohomology class of $\bar\partial\eta$. But
\[
\bar\partial\eta_j = -\frac{\bar Q_j}{u_{1j}^2u_{2j}^2}dz_j^{-2}\bar dz_j.
\]
So $\bar\partial\eta$ is represented by the cohomology class of the smooth form 
\[
-\bar\caQ/\gamma_1\gamma_2\in\Gamma(\Sigma_c,\caE^{0,1}(\caL)).
\]
Finally, we show that this vanishes when it is $\bar\partial$-exact, by showing that it is harmonic with 
respect to Hermitian metric $B=\gamma_1\gamma_2$ on $\caL$. With respect to
the Hodge inner product on $\caE^{*,*}(\caL)$ determined by $B$ on $\caL$ and the induced metric $\gamma$ 
on $\Sigma_c$, the adjoint of $\bar\partial:\caE^0(\caL)\to\caE^{0,1}(\caL)$ is given by 
\[
\caE^{0,1}(\caL)\stackrel{\bar\partial^*}{\to}\caE^0(\caL);\quad
\bar\partial^* = -\bar\star\bar\partial\bar\star,
\]
where $\bar\star$ is the conjugate linear Hodge star operator for our choice of metrics (see, for example, \cite[p168]{Wel}). 
It therefore suffices for us to show that 
\[
\bar\partial\bar\star(\bar\caQ/\gamma_1\gamma_2)=0.
\]
Let $(U,z)$ be a chart in which the divisor $D_1+D_2$ has at most one point, of degree $d$ at $z=0$.
Then $\tau= z^{-d}dz^{-2}$ is a local holomorphic trivialising section of $\caL$ over $U$ and
\[
\caE^{0,1}(U,\caL)\stackrel{\bar\star}{\to}\caE^{1,0}(U,\caL^*);\quad a\tau d\bar z\mapsto -i\bar aB(\cdot,\tau)dz,
\]
for any locally smooth complex function $a$ over $U$. Now in $U$ we write 
\[
\bar\caQ/\gamma_1\gamma_2 = \frac{\bar Q z^d}{u_1^2u_2^2}\tau d\bar z,
\]
and therefore
\[
\bar\partial\bar\star(\bar\caQ/\gamma_1\gamma_2) = 
-i\frac{\partial}{\partial\bar z}\left(\frac{Q\bar z^d\|\tau\|^2}{u_1^2u_2^2}\right)(z^d dz^2) d\bar z\wedge dz.
\]
But
\[
\|\tau\|^2 = B(\tau,\tau) = |z|^{-2d}u_1^2u_2^2,
\]
and therefore $\bar\partial\bar\star(\bar\caQ/\gamma_1\gamma_2)=0$ throughout $U$ since $Q$ is holomorphic.
\end{proof}
We finish this section by giving a global expression which relates the curvatures of the immersion $f$ to the norms of
$\caQ$ and $\caQ/\gamma_1\gamma_2$ with respect to the induced metric. This in turn provides an area bound for such
immersions.

First we note that a straightfoward calculation using \eqref{eq:nabla} shows that the projective flatness of the 
connexion $\nabla$ is equivalent to the local equations
\begin{align}
Z\bar Z\log u_1^2 & = 2u_1^2 + |Q|^2/u_1^2u_2^2 - u_2^2,\label{eq:Toda1}\\
Z\bar Z\log u_2^2 & = 2u_2^2 + |Q|^2/u_1^2u_2^2 - u_1^2,\label{eq:Toda2}
\end{align}
in a chart $(U,z)$ containing no complex or anti-complex points. 
These are the appropriate version of the Toda equations for this geometry (cf.\ \cite{BolW} for the $\CP^n$ version). They 
are also related to the two equations Wolfson derived for the \emph{\Kah angle} in \cite{Wol89}. Recall that the 
\Kah angle is
a continuous function $\theta:\Sigma_c\to\R$ for which $f^*\omega = \cos(\theta)v_\gamma$, where $v_\gamma$ is the area
form for the induced metric. Wolfson showed that, except at complex or anti-complex points where $\theta$ is not
differentiable, 
\begin{align}
i\partial\bar\partial \log\tan^2(\theta/2) & = f^*\Ric,\label{eq:tan}\\
i\partial\bar\partial \log\sin^2(\theta) & = (\kappa_\gamma + \kappa^\perp)v_\gamma,\label{eq:sin}
\end{align}
where $\Ric$ is the Ricci form and $\kappa_\gamma,\kappa^\perp$ are the Gaussian and normal curvatures of the immersion 
respectively. In our case
we have, from \eqref{eq:gamma}, $\cos(\theta) = (u_1^2-u_2^2)/(u_1^2+u_2^2)$. Therefore
\[
\tan^2(\theta/2) = \frac{u_2^2}{u_1^2},\quad \sin^2(\theta) = \frac{4u_1^2u_2^2}{(u_1^2+u_2^2)^2},
\]
and \eqref{eq:tan} is just the difference of \eqref{eq:Toda1} and \eqref{eq:Toda2} since $f^*\Ric =-6 f^*\omega$. 
Now 
\begin{align*}
Z\bar Z\log\sin^2\theta & = Z\bar Zu_1^2 + Z\bar Zu_2^2 - 2Z\bar Z\log(u_1^2+u_2^2)\\
 = & u_1^2 + \frac{2|Q|^2}{u_1^2u_2^2} + u_2^2 - \frac{2}{u_1^2+u_2^2}[Z\bar Z\log  (u_1^2+u_2^2)](u_1^2+u_2^2),
\end{align*}
using \eqref{eq:Toda1} and \eqref{eq:Toda2}. The second term on the right contains the local expression for $\kappa_\gamma$,
so substituting this equation into \eqref{eq:sin} reveals that 
\[
\kappa^\perp -\kappa_\gamma = 2\left(1+\frac{2|Q|^2}{u_1^2u_2^2(u_1^2+u_2^2)}\right).
\]
The right hand side can be written in terms of the quantities
\[
\|\caQ\|_\gamma = \frac{2\sqrt{2}|Q|}{(u_1^2+u_2^2)^{3/2}},\quad \|\frac{\caQ}{\gamma_1\gamma_2}\|_\gamma = 
\frac{|Q|\sqrt{u_1^2+u_2^2}}{\sqrt{2}u_1^2u_2^2}.
\]
It is easy to check that their product is smooth everywhere, so that we obtain the global identity
\begin{equation}\label{eq:curvatures}
\kappa^\perp -\kappa_\gamma = 2(1+\|\frac{\caQ}{\gamma_1\gamma_2}\|_\gamma\|\caQ\|_\gamma).
\end{equation}
By integration over $\Sigma$, and using \eqref{eq:chi},  we arrive at the following conclusion.
\begin{prop}\label{prop:area}
Let $f$ be a $\rho$-equivariant minimal immersion which is neither holomorphic nor anti-holomorphic, with induced
metric $\gamma$, cubic holomorphic differential $\caQ$, $d_1$ anti-complex points and $d_2$ complex points. Then
\begin{equation}\label{eq:Qbounds}
(4(g-1) -d_1-d_2)\pi \geq \Area_\gamma(\Sigma) +\int_\Sigma\|\caQ\|^2_\gamma v_\gamma,
\end{equation}
with equality if and only if either $\caQ=0$ or when $f$ is Lagrangian.
\end{prop}
Note that the stability inequalities \eqref{eq:stab} confirm that the left hand side is positive. 
The last statement follows because if $\caQ\neq 0$ equality requires $\|\caQ/\gamma_1\gamma_2\|_\gamma = \|\caQ\|_\gamma$,
which in turn requires $u_1=u_2$, whence $\cos(\theta)=0$.
\begin{rem}
The local equations \eqref{eq:Toda1}, \eqref{eq:Toda2} are clearly invariant under any unimodular scaling
$Q\mapsto e^{i\alpha}Q$.  Globally this corresponds to the symmetry $\caQ\mapsto e^{i\alpha}\caQ$, and by Theorem
\ref{thm:xi} this corresponds in turn to $\xi\mapsto e^{-i\alpha}\xi$. In fact this is equivalent to the well-known symmetry
of the equations \eqref{eq:projflat} $\Phi\mapsto e^{i\psi}\Phi$ (taking $\alpha=2\psi$) which Hitchin showed is a 
Hamiltonian circle action on the 
moduli space of $SL(2,\C)$-Higgs bundles \cite{Hit87}. To see this equivalence it is enough to perform the following gauge
transformation for $\partial_E+e^{i\psi}\Phi$ using the local gauge \eqref{eq:nabla}:
\begin{eqnarray}
& & \Ad \begin{pmatrix}  e^{-2i\psi} 0&0\\ 0&1&0\\ 0&0&e^{-i\psi}\end{pmatrix}\cdot
\left(\frac{\partial}{\partial z} +
\begin{pmatrix} Z\log u_1 & 0 & e^{i\psi}u_1\\ Q/u_1u_2 &  -Z\log u_2 & 0\\ 0 & e^{i\psi}u_2 & 0
\end{pmatrix} \right) \notag \\
&= &
\frac{\partial}{\partial z} +
\begin{pmatrix} Z\log u_1 & 0 & u_1\\ e^{2i\psi}Q/u_1u_2 &  -Z\log u_2 & 0\\ 0 & u_2 & 0
\end{pmatrix} 
\end{eqnarray}
Note in particular that, unlike the $SL(2,\C)$ case, the map $(E,\Phi)\mapsto (E,-\Phi)$ fixes every $PU(2,1)$-Higgs bundle since
$\partial_E+\Phi$ and $\partial_E-\Phi$ are gauge equivalent (by the symmetric space involution). 
\end{rem}

\section{Minimal Lagrangian immersions.}\label{sec:minlag}

By Wolfson's theorem \cite[Thm 2.1]{Wol89} Theorem \ref{thm:minimal} yields minimal Lagrangian immersions when both divisors
$D_1,D_1$ are zero.  Therefore Theorem \ref{thm:minimal} implies the following characterisation of all equivariant 
minimal Lagrangian immersions.  
\begin{cor}\label{cor:minLag}
Given a closed oriented surface $\Sigma$ of hyperbolic type, minimal Lagrangian immersions $f:\caD\to\CH^2$ which are
equivariant with respect to an irreducible representation of $\pi_1\Sigma$ into $PU(2,1)$ are in one-to-one correspondence with
the Higgs data $(\Sigma_c,\xi)$, where $c$ is a point in Teichm\"{u}ller space and $\xi\in H^{0,1}(\Sigma_c,K^{-2})$.
\end{cor}
It is not strictly necessary to say that $\rho$ is irreducible in this statement: this is implied by a combination of
Corlette's result that twisted harmonic maps correspond to reductive representations \cite{Cor} and Remark \ref{rem:hol}.

For a Lagrangian immersion the \Kah angle satisfies $\sin\theta=1$, so that \eqref{eq:sin} implies $\kappa^\perp=-\kappa_\gamma$.
Moreover, $\|\caQ\|^2_\gamma = 2\|A_f\|^2$, where $A_f$ is the shape operator of $f$ \cite[Lemma 2.8]{LofMsurv} and 
\[
\|A_f\|^2 = \sup\{\tfrac12 (\tr_\gamma A_f(\nu)^2):\nu\in T_p\caD^\perp,\ |\nu|=1\}.
\]
Since $\CH^2$ has constant Lagrangian sectional curvature $-1$, \eqref{eq:curvatures} reduces to the Gauss 
(and Ricci) equation for minimal Lagrangian immersions:
\begin{equation}\label{eq:LagGauss}
-1 = \kappa_\gamma + 2\|A_f\|^2.
\end{equation}
In \cite{LofM13} we wrote this as an equation for the conformal factor $\gamma=e^u\mu$ of the induced metric with
respect to the hyperbolic metric $\mu$:
\begin{equation}\label{eq:minlag}
\Delta_\mu u-2\|\caQ\|_\mu^2e^{-2u} - 2e^u +2=0.
\end{equation}
We gave existence results for this in terms of pairs $(\Sigma_c,\caQ)$,
and showed that there is a constant $k$, independent of $c\in\caT_g$, for which 
$\|\caQ\|_\mu\leq k$ yields a minimal $\rho$-equivariant embedding for which the normal bundle exponential map
$\exp^\perp:T\caD^\perp\to\CH^2$ is a diffeomorphism. We then showed that $\rho$ is quasi-Fuchsian, since the image
under $\exp^\perp$ of a fundamental domain for $\pi_1\Sigma$ gives a globally finite fundamental domain for $\rho$. 
Taking $\caQ=0$ gives the $\R$-Fuchsian representations, i.e., those which factor through the canonical inclusion 
$SO(2,1)\to PU(2,1)$. It is therefore convenient to adopt here the following terminology (similar 
terminology is used in the study of minimal surfaces in $\RH^3$, where it was introduced in \cite{KraS}).
\begin{defn}
A representation $\rho\in\Hom(\pi_1,PU(2,1))$ will be called \emph{almost $\R$-Fuchsian} if it admits a $\rho$-equivariant
minimal Lagrangian embedding $f:\caD\to\CH^2$ whose normal bundle exponential map $\exp^\perp$ is a diffeomorphism.
We will call $f$ an \emph{almost $\R$-Fuchsian embedding}.
\end{defn}
A necessary and sufficient condition for $\exp^\perp$ to be an immersion is that
$\|\caQ\|^2_\gamma\leq 2$ \cite[Thm 7.1]{LofM13}. The following theorem improves
substantially on the existence results in \cite{LofM13} by showing that almost $\R$-Fuchsian immersions exist, 
and are the unique minimal immersion, 
up to this optimal bound on $\caQ$: this is the direct analogue of Uhlenbeck's result \cite[Thm 3.3]{Uhl} for 
almost Fuchsian minimal surfaces in $\RH^3$.
\begin{thm}\label{thm:unique}
Let $f$ be a $\rho$-equivariant minimal Lagrangian immersion for which $\|\caQ\|^2_\gamma < 2$. Then $f$ is an almost
$\R$-Fuchsian embedding (and $\rho$ is almost $\R$-Fuchsian). Moreover, $f$ is the unique $\rho$-equivariant minimal immersion.
\end{thm}
\begin{rem}
This theorem should also be compared with the parametrisation of hyperbolic affine sphere immersions 
$\caD\to\R^3$ equivariant with respect to an irreducible representation into $PSL(3,\R)$, or equally, 
to the parametrisation of all convex real projective structures on $\Sigma$ \cite{Wan,Lof01,Lab07}. 
The data is a pair $(\Sigma_c,\caQ)$ where $\caQ$ is a cubic holomorphic differential. By a theorem of
Choi \& Goldman \cite{ChoG} this data parametrises
an entire component, the Hitchin component, of the representation space $\caR(SL(3,\R))$.
Labourie \cite{Lab07} directly related the hyperbolic affine sphere data to the Higgs bundles identified by 
Hitchin in \cite{Hit92}. In that case the Hitchin component is parametrized by Higgs bundles over $\Sigma_c$ with bundle
$K^{-1} \oplus 1 \oplus K$ and Higgs field 
\[ 
\begin{pmatrix} 0&1&0 \\ 0&0&1 \\ \mathcal Q & 0&0 \end{pmatrix}.
\]
Theorem \ref{thm:unique} shows that Corollary \ref{cor:minLag} provides a similar parametrisation of the almost Fuchsian locus 
inside the $\tau=0$ component $\caR(PU(2,1))$.  It is unlikely, however, to parametrise the entire component.  In the
analogous case of equivariant minimal surfaces in $\RH^3$, there are examples of quasi-Fuchisan hyperbolic 3-manifolds which
admit more than one minimal surface \cite{And,HuaW}.
\end{rem}
We will break the proof of Theorem \ref{thm:unique} into two parts, starting with the proof that $f$ is an embedding. 
For this we need to
recall from \cite[\S 7]{LofM13} an explicit expression for $\exp^\perp:T\caD^\perp\to\CH^2$ when $f$ is minimal Lagrangian. In this
case the frame $f_1,f_2,f_0$ used in the proof of Theorem \ref{thm:xi} has the simple form
\[
f_1 = \frac{1}{s}(f_0)_z,\quad f_2 = \frac{1}{s}(f_0)_{\bar z},\quad s = u_1=u_2.
\]
Let $S_-\subset \C^{2,1}_-$ be the pseudo-sphere into which $f_0$ maps, and $\pi:S_-\to\CH^2$ be the projection. 
When $f$ is Lagrangian $T\caD^\perp = Jf_*T\caD$ is a Lagrangian $2$-plane in $T\CH^2$. 
This implies that $\exp^\perp(T_z\caD^\perp)$ is a totally geodesic Lagrangian disc normal to $f(\caD)$ at $f(z)$. 
This Lagrangian disc has the form
\[
\{\pi[(ia-b)f_1(z) +(ia+b)f_2(z) + f_0(z)]:a^2+b^2< \tfrac12\}.
\]
Let $\Delta\subset\C$ denote the open disc of radius $1/2$.
Then we have an identification between $T\caD^\perp$ and $\caD\times\Delta$ for which $\exp^\perp$ is represented by
the map
\[
\Theta :\caD\times\Delta\to\CH^2;\quad \Theta(z,w) = \pi(-\bar w f_1(z) + wf_2(z) +f_0(z)).
\]
The following result improves \cite[Thm 8.1]{LofM13}.
\begin{lem}\label{lem:qF}
When $\|\caQ\|^2_\gamma< 2$ the pullback metric $\Theta^*g$ is complete, and therefore $\Theta$ is a proper map and 
hence a diffeomorphism. In that case $f$ is an embedding and $\rho$ is almost $\R$-Fuchsian.
\end{lem}
\begin{proof} Fix a point $p\in\caD$. We can normalise the frame $f_1,f_2,f_0$ so that these are the standard basis vectors
$e_1,e_2,e_3$ at $p$, and choose a conformal normal coordinate $z$ centred at $p$ so that $\gamma(p)=|dz|^2$ (i.e., 
$s(0)=1/\sqrt{2}$) with $s_z=0=s_{\bar z}$ at this point. We may also rotate $z$ so that $Q_0=Q(p)$ is real and
non-negative. Since $\caQ=Q_0dz^3$ and $\|dz\|=\sqrt{2}$ we have $0\leq Q_0<\tfrac12$. With such choices, in 
\cite[\S 7]{LofM13}, we computed the differential of $\Theta$ (in affine coordinates) at $p$ to be given by
\[
\begin{pmatrix} d\Theta_1\\ d\bar\Theta_1\\ d\Theta_2\\ d\bar\Theta_2\end{pmatrix}
=
\begin{pmatrix} l & k & 0 &-1\\ \bar k & l & -1 & 0 \\ \bar k & l & 1 & 0 \\ l & k & 0 & 1\end{pmatrix}
\begin{pmatrix} dz \\ d\bar z \\ dw \\ d\bar w \end{pmatrix},
\]
where
\[
l = \tfrac{1}{\sqrt{2}}(1+|w|^2),\quad k = -2Q_0w-\tfrac{1}{\sqrt{2}}\bar w^2.
\]
Now set $\phi = ldz+kd\bar z$ and notice that in affine coordinates $\Theta(p)=(-\bar w,w)$. Then at $p$ we compute the 
pull-back of $g$ to be
\begin{align}
(\Theta^*g)_p= &  \sum_{i,j=1}^2 \frac{1}{1-\|\Theta\|^2}\left(\delta_{ij} + 
\frac{\bar\Theta_i\Theta_j}{1-\|\Theta\|^2}\right)d\Theta_id\bar\Theta_j\notag \\
= &  \frac{1}{1-2|w|^2}\left[ \left(1+\frac{(-w)(-\bar w)}{1-2|w|^2}\right)(\phi-d\bar w)(\bar\phi - dw)\right. \notag \\
& + \left(0+\frac{(-w) w}{1-2|w|^2}\right)(\phi-d\bar w)(\bar\phi + d\bar w)\notag \\
& + \left(0+\frac{\bar w(-\bar w)}{1-2|w|^2}\right)(\bar\phi+dw)(\bar\phi - dw)\notag \\
&  +\left.\left(1+\frac{\bar ww)}{1-2|w|^2}\right)(\bar\phi+dw)(\phi + d\bar w)\right]\notag \\
=& \frac{1}{(1-2|w|^2)^2}\left([2|\phi|^2 -(w\phi+\bar w\bar\phi)^2] +[2|dw|^2 + (wd\bar w-\bar w dw)^2]\right). 
\label{eq:Theta*g}
\end{align}
Now consider the two terms in this expression:
\[
\theta_1 = \frac{1}{(1-2|w|^2)^2}\left(2|\phi|^2 -(w\phi+\bar w\bar\phi)^2\right),\quad
\theta_2 = \frac{1}{(1-2|w|^2)^2}\left(2|dw|^2+ (wd\bar w-\bar w dw)^2\right).
\]
The term $\theta_2$, which is the induced metric on $\Delta$, is just the Klein model for the hyperbolic plane and reflects
the fact that the fibres of $\exp^\perp$ are totally geodesic. The first term $\theta_1$ is the expression at $p$ for the
metric induced by the immersion $\varphi_w:\caD\to\CH^2$, $\varphi_w(z) = \Theta(z,w)$. We can think of each vector 
$(-\bar w,w)$ as determining a section of $T\caD^\perp$, and $\varphi_w$ is the image of this under $\exp^\perp$. We claim
that there is a constant $\varepsilon_1>0$ for which, for every $w$, 
\[
\varphi_w^*g(X,X)\geq \varepsilon_1\gamma(X,X),\quad\forall X\in T_p\caD.
\]
It follows that $\Theta^*g=\theta_1+\theta_2$ is bounded below by $\varepsilon_2\gamma + \theta_2$. The latter is a product
of complete metrics on $\caD\times\Delta$ and therefore $\Theta^*g$ is also complete. 

To prove the claim, write $w=w_1+iw_2$ and
$\phi=\phi_1+i\phi_2$ for real and imaginary parts, and set $r^2=|w|^2<1/2$. Then
\begin{equation}\label{eq:theta1}
\theta_1 = \frac{1}{(1-2r^2)^2}\begin{pmatrix} \phi_1 & \phi_2\end{pmatrix}\begin{pmatrix} 2-4w_1^2 & 4w_1w_2\\
4w_1w_2 & 2-4w_2^2\end{pmatrix} \begin{pmatrix} \phi_1 \\ \phi_2\end{pmatrix}.
\end{equation}
The eigenvalues of the matrix are $2$ and $2-4r^2$. Therefore, using the smaller eigenvalue $2-4r^2$,
\[
\theta_1 \geq \frac{2}{1-2r^2}|\phi|^2\geq 2|\phi|^2.
\]
It now suffices to show that 
\[
|\phi|^2\geq\varepsilon_2\gamma = \varepsilon_2(dx^2+dy^2),
\]
for a constant $\varepsilon_2>0$ independent of $w$, where $z=x+iy$. For this, write $k=k_1+ik_2$ so that
\begin{equation}\label{eq:B}
\begin{pmatrix} \phi_1\\ \phi_2 \end{pmatrix} = B \begin{pmatrix} dx \\ dy\end{pmatrix}\text{ for }
B=\begin{pmatrix} l+k_1& k_2\\k_2 & l-k_1\end{pmatrix}.
\end{equation}
The components of the metric $|\phi|^2$ are the entries of $B^tB=B^2$, and the eigenvalues are the roots of
\[
\lambda^2 - 2(l^2+|k|^2)\lambda +(l^2-|k|^2)^2=0,
\]
and are therefore $(l\pm|k|)^2$. Thus for any unit vector $X\in T_p\caD$,
\[
|\phi|^2(X,X)\geq (l-|k|)^2.
\]
Now
\[
l-|k| = \frac{l^2-|k|^2}{l+|k|},
\]
and $l+|k|$ is clearly bounded above, so it suffices to show that $l^2-|k|^2$ is bounded below by a positive constant 
independent of $w$. We compute
\begin{align}
l^2-|k|^2 & = \tfrac12(1+r^2)^2 - (2Q_0w+\tfrac1{\sqrt{2}}\bar w^2)(2Q_0\bar w+\tfrac1{\sqrt{2}}w^2) \notag \\
& = r^2(1 -4Q_0^2) +\tfrac12 - Q_0(\sqrt{2}r)^3\cos(3\alpha) \label{eq:l2-k2},
\end{align}
where $w=re^{i\alpha}$.
Now $Q_0<\tfrac12$ and $\sqrt{2}r<1$, so $1-4Q_0^2>0$ and $\tfrac12-2\sqrt{2}Q_0r^3\cos(3\alpha)>0$. With $\caQ$ fixed we get a
uniform positive lower bound over $r,\alpha$. Thus the metric $\Theta^*g$ is complete on $\caD\times\Delta$. We conclude, as
in \cite{LofM13}, that $\Theta$ is a proper map and a diffeomorphism, whence $f$ is an embedding and $\rho$ is almost
$\R$-Fuchsian.
\end{proof}
To prove that $f$ is unique we first need a result which can be given in greater generality than our current situation. 
Let $(N,g)$ be a complete Riemannian manifold. 
\begin{prop}\label{prop:unique}
Let $f:M\to (N,g)$ be a compact embedded minimal submanifold for which 
$\exp^\perp:TM^\perp\to N$ is a diffeomorphism. For a local section $\eta$ of
$TM^\perp$ of unit length and a positive constant $r$, set $\varphi_r=\exp^\perp(r\eta)$ and
let  $\vol_r$ be the volume form for the metric $\varphi^*_rg$ on an open subset of $M$.
Suppose $d\vol_r/dr>0$ for all $r$ and for every $\eta$. Then $f$ is the unique minimal immersion of $M$
transverse to the fibres of $\exp^\perp$.  
\end{prop}
The proof is given in Appendix \ref{sec:unique}. 

Now we can complete the proof of Theorem \ref{thm:unique}. Since $\rho$ is quasi-Fuchsian the quotient $\CH^2/\rho$ is a
manifold, and by the previous lemma $f:\Sigma\to\CH^2/\rho$ is a minimal embedding such that 
$\exp^\perp:T\Sigma^\perp\to \CH^2/\rho$ is a diffeomorphism. 
Now if $\varphi:\caD\to\CH^2$ is any $\rho$-equivariant immersion then it must be transverse to the fibres of 
$\exp^\perp$, 
since $d\varphi\circ d\delta = d\rho(\delta)\circ d\varphi$ for every $\delta\in\pi_1\Sigma$ and the action of $\rho$ is 
transverse to the fibres. Therefore the uniqueness claim in Theorem \ref{thm:unique} follows from the previous
proposition and the following lemma.
\begin{lem}
Under the assumptions of Theorem \ref{thm:unique}, if $v_r$ is the area form for the metric $\varphi_r^*g$ induced 
by any local immersion of the form $\varphi_r=\exp^\perp(r\eta)$ for a local section $\eta$ of
$TM^\perp$ of unit length and a positive constant $r$, then $dv_r/dr>0$.
\end{lem}
\begin{proof}
Since $\varphi_r(z) = \Theta(z,re^{i\alpha})$ for some fixed $r$ and $\alpha$, the
induced metric $\varphi_r^*g$ at a point $p\in\Sigma$ is given by $\theta_1$ from the proof of Lemma \ref{lem:qF}.
Using the expression \eqref{eq:theta1} we can write 
\[
\varphi_r^*g = \frac{1}{(1-2r^2)^2}\begin{pmatrix} dx & dy \end{pmatrix} B^t
\begin{pmatrix} 2-4w_1^2 & 4w_1w_2\\ 4w_1w_2 & 2-4w_2^2\end{pmatrix}
B\begin{pmatrix} dx\\ dy\end{pmatrix},
\]
where $B$ is given by \eqref{eq:B}. The determinant of the matrix in the middle is $2(2-4r^2)$, so that 
$v_r = a(r)dx\wedge dy$ where
\begin{align*}
a(r) &= \frac{1}{(1-2r^2)^2} 2\sqrt{1-2r^2}(l^2-|k|^2)\\
&= \frac{1}{(1-2r^2)^{3/2}}\left( 1+2r^2(1-4Q_0^2)-4\sqrt{2}Q_0r^3\cos(3\alpha)\right),
\end{align*}
using \eqref{eq:l2-k2}. A calculation shows that
\[
\frac{da}{dr} = \frac{4r}{(1-2r^2)^{5/2}}\left[ (1+r^2)(1-4Q_0^2) + \tfrac32 (1-2\sqrt{2}Q_0r\cos(3\alpha))\right].
\]
Now $Q_0<\tfrac12$, $r<1/\sqrt{2}$ and $\cos(3\alpha)\leq 1$, so $da/dr>0$.
\end{proof}

\subsection{Families of solutions to the Gauss equation.}\label{sec:Gauss}
Theorem \ref{thm:unique}  shows that the norm $\|\caQ\|_\gamma$ gives control over uniqueness
of minimal Lagrangian immersions, but at present we have no clear way of relating it to the parametrisation by the 
extension class $\xi$. Moreover, since this norm depends upon the solution to \eqref{eq:minlag} it is difficult to control
\emph{a priori}. On the other hand, the
combined results of \cite{LofM13} and \cite{HuaLL} show that a bound on $\|\caQ\|_\mu$ must be combined with a condition
the solution of \eqref{eq:minlag} to get existence and uniqueness. 
One knows that the zero solution $u\equiv 0$ is the unique solution for $\caQ=0$, and that for $\|\caQ\|_\mu$ small and
non-zero there are always solutions \cite{LofM13}, but these are not unique \cite{HuaLL}. 
The challenge is to 
understand how solutions behave as one moves along a ray $t\caQ_0$, $t\geq 0$, given a fixed cubic holomorphic
differential $\caQ_0$. 
To study solutions along such rays, Huang et.\ al \cite{HuaLL} introduced the following terminology.
\begin{defn}
A solution $u$ to \eqref{eq:minlag} is \emph{$\caF$-stable} if the linearised operator
\begin{equation}\label{eq:lin}
\caL = -\Delta_\mu +2e^u-4\|\caQ\|_\mu^2e^{-2u},
\end{equation}
is positive.
\end{defn}
This condition ensures, by the Implicit Function Theorem in the appropriate Sobolov spaces, that there is locally a 
smooth curve $u(t)$ of solutions to
\begin{equation}\label{eq:family}
\caH(u,t) = \Delta_\mu u-2\|t\caQ_0\|_\mu^2e^{-2u} - 2e^u +2=0,
\end{equation}
nearby any $\caF$-stable solution $u(t_0)$. Our aim here is to show that, given $\caQ_0$, the $\caF$-stable solutions form
a continuous curve terminated at one end by the zero solution and at the other end by the first
solution which is not $\caF$-stable. This, together with the results of \cite{LofM13,HuaLL}, gives the following summary
of the behaviour of solutions to \eqref{eq:minlag} as the cubic differential is scaled.
\begin{thm}\label{thm:rays}
Fix a non-zero cubic holomorphic differential $\caQ_0$ on $\Sigma_c$. Set 
\[
T_0=\sqrt{4/27}(\sup_\Sigma\|\caQ_0\|_\mu)^{-1}.
\] 
Then:
\begin{enumerate}
\item there exists a $T_2>T_0$ such that \eqref{eq:family} has no solutions for $t\geq T_2$; 
\item there exists $T_0\leq T_1< T_2$ such that for $t<T_1$ there is a unique continuous family of $\caF$-stable solutions to 
\eqref{eq:family}. All $\caF$-stable solutions lie on this family and the limiting solution $u(T_1)$ is not
$\caF$-stable;
\item for $t<T_0$ the $\caF$-stable solutions yield almost Fuchsian embeddings; 
\item for $0<t<T_1$ there is at least one solution which is not $\caF$-stable.
\end{enumerate}
\end{thm}
\begin{rem}
This result is analogous to Uhlenbeck's description of the bifurcation in families of solutions to the Gauss equation for
minimal surfaces in $\RH^3$ \cite[Thm 4.4]{Uhl}. We note that for Uhlenbeck the right notion of stability was stability with
respect to the area functional. In our case that gives no extra control, since all minimal Lagrangian surfaces in
$\CH^2/\rho$ are stable by a theorem of Oh \cite{Oh}.
\end{rem}
Part (i) comes from \cite{LofM13}, while (iii) comes from \cite{LofM13} and Theorem \ref{thm:unique} above. Part (iv) and the 
existence of a local family of unique $\caF$-stable solutions
for $t<T_1$ come from \cite{HuaLL}. Here we provide the rest of (ii) via the following lemma.
\begin{lem}
Let $\caQ_0$ be a holomorphic cubic differential on $\Sigma_c$, and let $\tau>0$ be such that $u(\tau)$ is an
$\caF$-stable solution. Let $u(t)$ be the local family of $\caF$-stable solutions to
\eqref{eq:family} through $u(\tau)$. Then $\dot u(\tau)\leq 0$ on all of $\Sigma_c$.
\end{lem}
\begin{proof}
By differentiating \eqref{eq:family} we find that $\dot u $ satisfies
\[
-\caL_\tau \dot u =4\tau\|\caQ_0\|_\mu^2 e^{-2u}.
\]
Elliptic regularity implies $\dot u$ is $C^\infty$.  Now define $\dot u^+ = \max\{\dot u(\tau),0\}$.  

Then $\dot u^+$ is in the Sobolev space $H(\Sigma)$ and $d\dot u^+=0$ wherever $\dot u\le0$ (see e.g.\cite{GilT}).  Recall we
define
$$
\langle -\Delta_\mu \dot u^+, \dot u^+\rangle = \int_\Sigma \|d\dot u^+\|^2_\mu\,dA_\mu
$$
in this case. Let  $v = 4\tau \|\mathcal Q_0\|_\mu^2 e^{-2u}$, and let $\epsilon_j\searrow 0$ be regular values of $\dot u$
(as guaranteed by Sard's Theorem). Thus we can integrate by parts, as each $\{\dot u =\epsilon_j\}=\partial\{\dot u
>\epsilon_j\}$ is a smooth 1-manifold. Let $w_j = \max\{\dot u-\epsilon_j,0\}$. Now since $\mathcal L>0$,
\begin{eqnarray*}
0 &\ge& \langle -\mathcal L \dot u^+,\dot u^+\rangle \\
&=& \int_{\Sigma} \left[ -\|d\dot u^+\|_\mu^2 +(-2e^u +4 \|\mathcal Q_0\|_\mu^2 e^{-2u})(\dot u^+)^2 \right] \,dA_\mu \\
&=& \int_{\{\dot u>0\}} \left[ -\|d\dot u\|_\mu^2 +(-2e^u +4 \|\mathcal Q_0\|_\mu^2 e^{-2u})\dot u^2 \right] \,dA_\mu \\
&=& \lim_{j\to\infty}\int_{\{\dot u>\epsilon_j\}} \left[ -\|d w_j\|_\mu^2 +(-2e^u +4 \|\mathcal Q_0\|_\mu^2 e^{-2u})\dot
u^2 \right] \,dA_\mu \\
&=& \lim_{j\to\infty}\int_{\{\dot u>\epsilon_j\}} \left[ w_j\Delta_\mu w_j +(-2e^u +4 \|\mathcal Q_0\|_\mu^2 e^{-2u})\dot
u^2 \right] \,dA_\mu \\
&=& \lim_{j\to\infty}\int_{\{\dot u>\epsilon_j\}} \left[ (\dot u -\epsilon_j) \Delta_\mu \dot u  +(-2e^u +4 \|\mathcal
Q_0\|_\mu^2 e^{-2u})\dot u^2 \right] \,dA_\mu \\
&=& \int_{\{ \dot u>0\}} (-\mathcal L \dot u)\dot u \,dA_\mu \\
&=& \int_{\{ \dot u>0\}} v\dot u \,dA_\mu \ge0. \\
\end{eqnarray*}
(The limits above are valid by the Dominated Convergence Theorem.) 

Since $v$ is positive almost everywhere, the last inequality is strict if $\dot u>0$ anywhere on $\Sigma$. Thus
by contradiction $\dot u(\tau)\le 0$ everywhere on $\Sigma$.
\end{proof}
\begin{proof}[Proof of Thm \ref{thm:rays}(ii)] 
Let $u(\tau)$ be an $\caF$-stable solution, with local family $u(t)$ and let $\caL_t$ be the corresponding family of
linearised operators \eqref{eq:lin}. Now
\[
\dot\caL = 2\dot{u}e^u + 8t\|\caQ\|^2_\mu e^{-2u}(\dot{u}-t),
\]
which is nonpositive for $t>0$ by the previous Lemma. Thus $\caL_\tau>0$ implies $\caL_t>0$ for all 
$t\in[0,\tau]$ in the path of solutions.

The Maximum Principle shows that every solution to \eqref{eq:family} is nonpositive.  Thus for any $t$ in an interval of
the form $(\tau_0,\tau]$ the proposition implies $0\ge u_t \ge u_\tau$, and thus we have uniform $L^\infty$ bounds on
$u_t$ for all $t\in(\tau_0,\tau]$.  Then the $L^p$ theory, applied to \eqref{eq:family}, and standard bootstrapping show
that the limit
\[
\lim_{t\to\tau_0^+} u_t
\]
exists and is a solution $u_{\tau_0}$ to the equation.  This provides a closedness argument for the continuity method.
On the other hand the $\caL_t>0$ condition, verified in the previous paragraph, provides openness as well, and thus
we can extend the solution space back down to $t=0$.
\end{proof}

\section{Surfaces with zero cubic holomorphic differential.}\label{sec:Q=0}

Minimal (possibly branched) immersions for which $\caQ=0$ have particularly important properties. They include all the
holomorphic and anti-holomorphic immersions and, by Theorem \ref{thm:xi}, the extension class zero cases when $f$ is not
holomorphic or anti-holomorphic. In all such cases, the Higgs bundle $E$ is a \emph{Hodge bundle} (or \emph{variation of Hodge
structure}), i.e., $E=\oplus_{i=1}^m E_i$ for proper sub-bundles $E_i$ for which $\Phi:E_i\to E_{i+1}\otimes K$, with
$E_{m+1}=0$. For $PU(2,1)$ the \emph{length} $m$ of the Hodge bundle must be either
two or three \cite{Got}. We will show below that the length-two
Hodge bundles correspond to holomorphic or anti-holomorphic immersions, while immersions which arise from Theorem
\ref{thm:xi} with $\xi=0$ give length-three Hodge bundles.

Hodge bundles play the central role in calculating the Betti numbers of the smooth components
of the moduli space $\caH(\Sigma_c,G)$ of polystable Higgs bundles, and therefore the Betti numbers of the
representation space $\caR(G)$. For on smooth components the \emph{Hitchin function}
$\|\Phi\|^2_{L^2}:\caH(\Sigma_c,G)\to\R$,
is a perfect Morse-Bott function, whose critical points are the Hodge bundles. The length-two Hodge bundles are minima, 
while the length-three Hodge bundles have non-zero Morse index \cite{Got}.

\subsection{Holomorphic and anti-holomorphic surfaces.}

By Toledo's theorem \cite{Tol} every \emph{maximal representation} (those for which $\tau(\rho)=\pm\chi(\Sigma)$) 
leaves invariant a complex line and acts on that line as a Fuchsian representation. Such representations 
are reducible. To understand the non-maximal $\rho$-equivariant holomorphic or anti-holomorphic immersions, we will start by
describing their Hodge bundles.  First we note that for
any holomorphic $\rho$-equivariant immersion $f:\caD\to\CH^2$ the area form $v_\gamma$ for the induced
metric equals $f^*\omega$. It follows from the definition \eqref{eq:taudef} that $\tau(\rho)>0$. 
For anti-holomorphic immersions $f^*\omega = -v_\gamma$ so that $\tau(\rho)<0$. The next lemma characterises
the Higgs bundle data for representations which admit either holomorphic or anti-holomorphic branched immersions.
\begin{thm}\label{thm:hol}
Each irreducible representation $\rho$ which admits a branched holomorphic $\rho$-equivariant immersion 
corresponds to a Hodge bundle $(V\oplus 1,\Phi)$ with $\Phi=(\Phi_1,0)$, where $V$ is given by a non-trivial
extension of the form
\begin{equation}\label{eq:holext}
0\to K^{-1}(B)\stackrel{\Phi_1}{\to} V \to K^{-2}L\to 0.
\end{equation}
Here $B$ is an effective divisor of degree $b\geq 0$ (the divisor of branch points of the immersion) and $L$ is a 
line bundle of degree $l$, only determined up to isomorphism. These satisfy the inequalities
\begin{equation}\label{eq:ineq}
3(g-1)+\tfrac12 b< l< 6(g-1)-b ,\quad 0\leq b< 2(g-1).
\end{equation}
In particular, $\tau(\rho) = \tfrac23(6g-6-b-l)$ and $0<\tau(\rho)<2(g-1)-b$. 

Moreover, $(E,\Phi)$ corresponds to a branched anti-holomorphic immersion $f$ if and only if $\bar f$ is the  branched
holomorphic immersion determined by $(E^*,\Phi^t)$.
\end{thm}
Note that by $\bar f$ we mean the post-composition of $f$ with the natural anti-holomorphic involution on $\CH^2$ which
descends from complex conjugation on $\C^{2,1}$. Clearly $f$ is $\rho$-equivariant precisely when $\bar f$ is
$\bar\rho$-equivariant. The map $\rho\mapsto\bar\rho$ is an involution on $\Hom(\pi_\Sigma,G)/G$ for which
$\tau(\bar\rho)=-\tau(\rho)$: it fixes the representations with values in $SO(2,1)$. 

The conditions in \eqref{eq:ineq} are only necessary conditions to ensure stability: there will be
extensions which do not provide stable Higgs bundles. As the proof below shows, stability requires the additional
condition that every line subbundle $F\subset V$ with $F\neq\Im(\Phi_1)$ has $\deg(F)<\tfrac13\deg(V)$.  
Nevertheless one can describe the structure of the moduli space of such stable extensions:
see \S\ref{subsec:caW} below.
\begin{proof}
First suppose $(E,\Phi)$, with $E=V\oplus 1$ and $\Phi=(\Phi_1,\Phi_2)$, is a length-two Hodge bundle with $\deg(V)<0$. 
In this case by \cite[\S 3]{Got} we have $\Phi_2=0$. The corresponding representation $\rho$ admits a $\rho$-equivariant 
harmonic map $f:\caD\to\CH^2$ determined by the sub-bundle $1$ as a section $\P E_-$ with $df' = \Phi$. Therefore
$\bar\partial f'=0$ and $f$ is holomorphic. Conversely, suppose $\rho$ admits a holomorphic immersion $f$.
Taking $V=f^{-1}T'\CH^2$ and $\Phi = \partial f':1\to KV$ gives a length-two Hodge bundle with $\deg(V) < 0$.  
The involution $(E,\Phi)\to (E^*,\Phi^t)$ on Higgs bundles maps length-two Hodge bundles with $\Phi_2=0$ to those with
$\Phi_1=0$. In the latter case $V\simeq f^{-1}T''\CH^2$ with the opposite complex structure, and the map $f$ is
anti-holomorphic. 

Now the structure of $(V\oplus 1,\Phi)$ follows a simplified version of the argument in the proof of Theorem
\ref{thm:minimal}. We can think of $\Phi_1$ as a holomorphic section of $KV$, with divisor $B  \geq 0$ corresponding to the
branch divisor of $f$. The
bundle injection $\Phi_1:K^{-1}(B  )\to V$ has image $V_1$ and quotient line bundle $V_2=V/V_1$. 
Provided $V$ is not the direct sum $V_1\oplus V_2$ the $\Phi$-invariant sub-bundles of $E$
are $V_1, V_1\oplus 1$, $V$ and any line subbundle $F\subset V$ with $F\neq V_1$. The stability inequalities 
are therefore
\[
\deg(V_1)<\tfrac13 \deg(V),\quad \tfrac12\deg(V_1)<\tfrac13\deg(V),\quad \tfrac12\deg(V)<\tfrac13\deg(V),\quad
\deg(F)<\tfrac13\deg(V).
\]
The inequalities \eqref{eq:ineq} are equivalent to the first three of these.
On the other hand, if $V$ is the direct sum then $V_2$ is also $\Phi$-invariant and stability
requires the additional inequality 
\[
\deg(V_2)<\tfrac 13\deg(V),\ \text{i.e.,}\quad \tfrac 23\deg(V) < \deg(V_1),
\] 
so this is not possible. For later convenience we write $V_2=K^{-2}L$ and the inequalities \eqref{eq:ineq} follow 
from $\deg(V_1)=b-2(g-1)$ and $\deg(V_2)=l-4(g-1)$. 
\end{proof}
Note that while the splitting of $V\simeq f^{-1}T'\CH^2/\rho$ into $T\Sigma\oplus T\Sigma^\perp$ is $J$-invariant, 
the sub-bundle $T\Sigma^\perp$ is not $\bar\partial_E$-invariant unless $\rho$ is reducible. Indeed, the normal bundle is
$\partial_E$-invariant (since it is paired with $T\Sigma$ by the Hermitan metric) 
so the induced structure of this splitting makes the normal bundle \emph{anti-holomorphic}. 
\begin{rem}[Reducible representations.]\label{rem:reducible}
Although $E =V\oplus 1$ cannot be stable when $V$ is a trivial extension, it can be polystable.
This corresponds to a reducible reductive representation. Such representations either:
\begin{enumerate}
\item factor through a maximal compact subgroup, or,
\item factor through $P(U(1,1)\times U(1))$.
\end{enumerate}
This is easy to see. We may simplify things, by replacing $\rho$ by $\bar\rho$ if necessary, to assume that
$\tau(\rho)\geq 0$ and thus $\Phi_2=0$.  To be strictly polystable $(E,\Phi)$ must decompose into either
\[
(i)~(V,0)\oplus (1,0),\ \text{or}\quad (ii)~(1\oplus V_1,\Phi_1)\oplus (V_2,0),
\]
where $V=V_1\oplus V_2$ and $\Phi_1:1\to V_1$. In the first case $V$ must be a stable rank two bundle of degree zero (to
have the same slope as $1$) and therefore the representation lies in a maximal compact subgroup and has $\tau(\rho)=0$. 
We note that the corresponding harmonic map $f$ is constant. In the second case $(1\oplus V_1,\Phi_1)$ corresponds to a
representation into $U(1,1)$. In this case polystability requires
\[
\deg(V_1)<\tfrac12\deg(V_1\oplus 1) = \tfrac12 \deg(V_1),\ \text{i.e.,}\quad \deg(V_1)<0,
\]
together with the ``same slope'' condition $\tfrac12\deg(V_1) = \deg(V_2)$.
When we write $V_1 = K^{-1}(B)$ as above we deduce that 
\[
b<2(g-1),\quad \deg(V_2)=\tfrac12 b- g+1.
\]
Thus $b$ is even and
\[
\tau(\rho) = -\tfrac23\deg(V) = -\deg(V_1) = 2g-2-b\in 2\Z.
\]
In particular, such $\rho$ can only admit an unbranched holomorphic map $f$ when $\tau(\rho)$ is maximal, 
i.e., when $\rho$ factors through a Fuchsian representation. The map $f:\caD\to\CH^2$ is a totally geodesic embedding onto a
complex line. 
More generally, the $PU(1,1)$ representation corresponding to the rank two Higgs bundle $(1\oplus V_1,\Phi_1)$ has
Toledo invariant $-\deg(\Hom(1,V_1))$, which therefore equals $\tau(\rho)$.  
Every irreducible representation into $PU(1,1)$ of even Toledo invariant lifts
to $SU(1,1)$, and therefore provides a representation into $P(U(1,1)\times U(1))$. Thus the whole structure of
Higgs bundles for irreducible representations in $SU(1,1)$ \cite{Hit87} lifts up to provide reducible representations into 
$PU(2,1)$, and this is what we are seeing above. Note that those which are non-maximal cannot be convex cocompact, since
they preserve a complex line but act non-cocompactly on this line.
\end{rem}
Let $\eta\neq 0$ be the extension class of the extension \eqref{eq:holext} for $\rho$ irreducible. Since $L$ is only determined up to 
isomorphism the Higgs bundle only determines
$\eta$ up to scale. Therefore each Higgs bundle in Theorem \ref{thm:hol} corresponds to a quadruple $(\Sigma_c,B,L,\C.\eta)$ where 
$\C.\eta$ is the line generated by $\eta$, i.e., a point in $\P H^1(\Sigma_c,KL^{-1}(B))$. In fact, it is not hard to show 
that the rescaling of $\eta$ corresponds to the $\C^*$-action on the Higgs bundle $(E,\Phi)\mapsto (E,t\Phi)$. Since Hodge bundles are invariant under
this action, this gives another way of interpreting the independence of the data on the scale of $\eta$. However, once $(\rho,f)$ is known there
is a preferred representative for $\eta$ given by the geometric invariants of $f$ 
via the Dolbeault isomorphism. First we need to introduce the tensor
\[
S\in\caE^0(\Sigma_c,K^2\bar K^2),\quad S = h(\II^{2,0},\II^{2,0}),
\]
where $\II^{2,0}=\pi^\perp\nabla'\partial f'$ is the $(2,0)$ component of the second fundamental form of $f$ (here
$\pi^\perp:f^{-1}T'\CH^2/\rho\to T\Sigma^\perp$ is projection onto the normal bundle). We will show that $\II^{2,0}$ is a 
holomorphic section of $K^2\otimes T\Sigma^\perp$.
\begin{thm}\label{thm:eta}
Let $f:\caD\to\CH^2$ be a $\rho$-equivariant branched holomorphic immersion, with $\rho$ irreducible and data 
$(\Sigma_c,B,L,\C.\eta)$. 
Then $L\simeq \caO(D)$, where $D$ is the divisor of zeroes of $\II^{2,0}$, and we can choose $\eta$ so that under
the Dolbeault isomorphism it maps to the cohomology class
\[
-[S/\gamma]\in H^{0,1}(\Sigma_c,K(B-D)).
\]
\end{thm}
\begin{proof}
We follow the steps in the proofs of Theorem \ref{thm:xi} but using a local frame more suited to holomorphic maps. As
before, use $\ell_0\subset E$ to denote $1$ and write its $\partial f'$ transform as $\ell_1\otimes K$. But now take the
further transform of $\ell_1$, so we have a harmonic sequence
\[
\ell_0\stackrel{\pi_0^\perp Z}{\to}\ell_1\stackrel{\pi_1^\perp Z}{\to}\ell_2\stackrel{\pi_2^\perp Z}{\to}0,
\]
in each chart $(U,z)$. The last step terminates the sequence because $\pi_2^\perp Z:\ell_2\to\ell_0$ is the adjoint to 
$\pi_0^\perp\bar Z:\ell_0\to\ell_2$, which 
represents $\bar\partial f'=0$. As before, let $f_0\in\Gamma(\ell_0)$ be global and parallel with $\g{f_0}{f_0}=-1$.
Set
\[
\sigma_1 = \pi_0^\perp Zf_0 = Zf_0,\quad \sigma_2 = \pi_1^\perp Z\sigma_1=\pi_1^\perp ZZf_0.
\]
Since $f_0$ is a holomorphic section of $\ell_0$, $\sigma_j$ is a holomorphic section of $\ell_j$
(by standard harmonic sequence theory \cite{BolW}).
Under the isomorphism $f^{-1}T'\CH^2/\rho\simeq\Hom(1,V)\simeq V$ the image $\partial f'(T^{1,0}\Sigma)$ of the holomorphic 
tangent bundle of $\Sigma$ is identified with $\ell_1$, i.e., $\ell_1\simeq K^{-1}(B)$. 
Clearly the induced metric of $f$ is $u_1^2|dz|^2$ for $u_1 = |\sigma_1|$.  Further, since
\[
(\nabla_Z Z)f_0 = \pi_0^\perp Z\pi_0^\perp Z f_0- \pi_0^\perp Z \pi_0 Zf_0=\pi_0^\perp ZZf_0,
\]
we have
\[
\II(Z,Z)=[\pi^\perp(\nabla_Z Z)]f_0 = \pi_1^\perp ZZf_0=\sigma_2.
\]
In particular, $\II^{2,0}$ is a holomorphic section of $K^2\otimes \ell_2$.
Set $s_2=|\sigma_2|$. Then 
\[
S(Z,Z,\bar Z,\bar Z) = \g{\II(Z,Z)}{\II(Z,Z)} = |\sigma_2|^2= s_2^2.
\]
Thus in a chart $U$ in which neither $u_1$ nor $s_2$ vanishes we have a local $U(2,1)$ frame given by $f_1,f_2,f_0$ where 
$f_1=\sigma_1/u_1$, $f_2 = \sigma_2/s_2$. Straighforward calculations as before give
\begin{align}\label{eq:Zhol}
Zf_1& = (Z\log u_1) f_1 + u_1^{-1}s_2f_2,\\
Zf_2& = (Z\log s_2) f_2,\notag \\
Zf_0& = u_1f_1.\notag
\end{align}
From this we can read off the connexion $1$-form for the projectively flat connexion $\nabla$ in this frame.
Now let $\varphi:V|U\to U\times\C^2$ be the local trivialisation corresponding to $f_1,f_2$, then it follows that
\[
\varphi\circ\bar\partial_E\circ\varphi^{-1} =d\bar z\left[ 
\frac{\partial}{\partial\bar z} - \begin{pmatrix} \bar Z\log u_1 &s_2/u_1\\ 0 &\bar Z\log s_2 \end{pmatrix}
\right].
\]
To deal with zeroes of $u_1$ and $s_2$, we may assume $U$ only has these at $z=0$, to order $p$ and $q$ respectively. In
such a chart we take the local frame $\tilde f_1 = w^{-p}f_1$, $\tilde f_2 = w^{-q}f_2$, where $w=z/|z|$, and in the
corresponding trivialisation $\tilde\varphi$ we have
\[
\tilde\varphi\circ\nabla^{0,1}\circ\tilde\varphi^{-1} =d\bar z\left[ 
\frac{\partial}{\partial\bar z} - \begin{pmatrix} \bar Z\log (u_1/|z|^p) &-s_2w^{p-q}/u_1\\ 0 &\bar Z\log (s_2/|z|^q) 
\end{pmatrix}
\right].
\]
We obtain a local holomorphic trivialisation, with respect to $\bar\partial_E$, by applying a gauge transformation
of the form
\[
R = \begin{pmatrix} 1& a\\ 0&1\end{pmatrix}\begin{pmatrix} |z|^p/u_1 & 0 \\ 0 & |z|^q/s_2\end{pmatrix},
\]
i.e., for $\chi = R\tilde\varphi$ we have $\chi\circ\bar\partial_E\circ\chi^{-1} = d\bar z(\partial/\partial\bar z)$.
This requires
\begin{equation}\label{eq:s/u}
\partial a/\partial\bar z =- s_2^2z^{p-q}/u_1^2.
\end{equation}
When $U$ is contractible this equation has a solution since $s_2^2/u_1^2\sim |z|^{2(q-p)}$ near $z=0$, therefore the 
right hand side of \eqref{eq:s/u} is smooth throughout $U$. 

Now cover $\Sigma_c$ by charts $(U_j,z_j)$ of the type used above, and index the local objects living over $U_j$ by $j$. 
Thus for $V$ we have transition relations $\tilde\varphi_j = \tilde c_{jk}\tilde\varphi_k$ where
\[
\tilde c_{jk} = \begin{pmatrix} \frac{dz_j/dz_k}{|dz_j/dz_k|}& 0 \\ 
0 & \frac{(dz_j/dz_k)^2}{|dz_j/dz_k|^2}\end{pmatrix}
\begin{pmatrix} w_j^{p_j}w_k^{-p_k} & 0 \\ 0 & w_j^{q_j}w_k^{-q_k}\end{pmatrix}.
\]
Therefore $\chi_j = b_{jk}\chi_k$ where 
\begin{equation}
b_{jk} = R_j\tilde c_{jk} R_k^{-1} = 
\begin{pmatrix} z_j^{p_j}z_k^{-p_k}dz_j/dz_k & \lambda_{jk}\\ 0 & z_j^{q_j}z_k^{-q_k} (dz_j/dz_k)^2 \end{pmatrix},
\end{equation}
for
\begin{equation}
\lambda_{jk} = a_jz_j^{q_j}z_k^{-q_k} (dz_j/dz_k)^2 - a_k z_j^{p_j}z_k^{-p_k}dz_j/dz_k,
\end{equation}
and we have used the fact that
\[
u_{1k}/u_{1j} = |dz_j/dz_k|,\quad s_{2k}/s_{2j} = |dz_j/dz_k|^2.
\]
In particular, this shows that $V$ is an extension of the line bundle $K^{-1}(B)$ by the line bundle $K^{-2}(D)$ where $D$ is the divisor of zeroes of
$\II^{2,0}$, and therefore we have fixed an isomorphism $L\simeq \caO(D)$.  
As earlier, the extension class of $V$ is given by the Cech cohomology class $\eta$ of the $1$-cocycle $\{(\eta_{jk},U_j,U_k)\}$ where 
\begin{align*}
\eta_{jk}& = z_j^{-p_j}z_k^{p_k} (dz_k/dz_j)\lambda_{jk} dz_k\\
& = a_jz_j^{q_j-p_j}dz_j - a_kz^{q_k-p_k}dz_k.
\end{align*}
This is plainly a coboundary for Cech cohomology in smooth sections of $K(B-D)$ of the form
$\delta\tau$ where $\{(\tau_j,U_j)\}$ has $\tau_j = a_jz^{q_j-p_j}dz_j$. Under the Dolbeault isomorphism this corresponds to 
\[
\bar\partial\tau_j = -\frac{s_{2j}^2}{u_{1j}^2}dz_jd\bar z_j,
\]
which is the local expression for $-S/\gamma$.
\end{proof}

\subsection{Surfaces arising from zero extension class.}
Theorems \ref{thm:minimal} and \ref{thm:xi} imply that Higgs bundles for non-holomorphic minimal surfaces with 
$\caQ=0$ are exactly the length-three Hodge bundles. When $\caQ=0$ we have a trivial extension bundle and can take
\[
E_1\oplus E_2\oplus E_3= K(-D_2) \oplus 1\oplus K^{-1}(D_1),
\] 
to get $\Phi:E_i\to KE_{i+1}$. Conversely, any length-three Hodge bundle is projectively equivalent to one of this form
and has $\tr\Phi^2=0$. We will show that $f$ is still related to a holomorphic map, via its harmonic 
sequence  (in the sense of Erdem \& Glazebrook \cite{ErdG}, $f$ is isotropic but non-holomorphic). 
To explain this, we first recall (from e.g., \cite{BolW,EelW}), the notion of the 
\emph{Gauss transforms} $\varphi_1,\varphi_2$ of $f$.

The line bundles $\ell_1,\ell_2$ defined by \eqref{eq:harmseq} both lie inside the subset $E_+$ consisting of fibre 
vectors which have positive length with respect to the $\C^{2,1}$ metric. Let 
\[
\C^{2,1}_+ = \{v\in\C^{2,1}:\g{v}{v}>0\}. 
\]
Then its projective space $\P\C^{2,1}_+$ is an open submanifold of $\CP^2$ on which
$G=PU(2,1)$ acts transitively. We identify it with the orbit of the line $[e_2]$, with isotropy subgroup 
$H_2\simeq P(U(1,1)\times U(1))$, so that $\P\C^{2,1}_+\simeq G/H_2$. 
In fact we can think of it as the complex version of two dimensional de Sitter space, and will henceforth denote it by 
$\CdS^2$.  Its tangent space at the base point $[e_2]$ is identified with the orthogonal
complement $\fm_2=\fh_2^\perp\subset \su(2,1)$ with respect to the Killing form $-\tfrac12\tr(AB)$, and the latter gives 
$\fm_2$ an indefinite Hermitian metric. This extends to the tangent bundle, isomorphic to $G\times_{H_2}\fm_2$, and makes 
$\CdS^2$ a pseudo-Hermitian symmetric space. 
Clearly $\P E_+\simeq \caD\times_\rho\CdS^2$ and therefore $\ell_1,\ell_2$ each determine a smooth $\rho$-equivariant map
$\varphi_1,\varphi_2:\caD\to \CdS^2$, and these will be conformal harmonic with respect to the pseudo-Hermitian metric
(they are isotropic in the sense of \cite{ErdG}). 
Following the terminology of harmonic sequences, we call $\varphi_1$ the 
\emph{$\partial$-Gauss transform} of $f$, and $\varphi_2$ the \emph{$\bar\partial$-Gauss transform} of $f$. 
An immersion into $\CdS^2$ is \emph{timelike} when its induced metric is negative definite (away from branch points).
\begin{prop}
Let $f:\caD\to\CH^2$ be $\rho$-equivariant and not $\pm$-holomorphic. Then $\caQ=0$ if and only if 
the $\bar\partial$-transform $\varphi_2:\caD\to\CdS^2$ of $f$ is a timelike $\rho$-equivariant holomorphic map, 
branched at the divisor of complex points $D_2$ of $f$.
\end{prop}
\begin{proof}
Let $\varphi_2:\caD\to \CdS^2$ be the $\bar\partial$-Gauss transform $f$. Write the differential of $\varphi_2$ 
as $d\varphi_2' = \partial\varphi_2'+\bar\partial\varphi_2'$ in terms of the type decomposition
\[
\varphi_2^{-1}T_{\ell_2}\CdS^2 = \Hom(\ell_2,\ell_2^\perp)\oplus\Hom(\ell_2^\perp,\ell_2).
\]
In local coordinates
\[
\bar\partial\varphi_2'(\bar Z) = \pi_2^\perp \bar Z.
\]
But from \eqref{eq:Qf0} and the fact that $\ell_0,\ell_1,\ell_2$ are mutually orthogonal we have
\begin{align*}
Q & = \g{\pi_0^\perp Z\pi_0^\perp Zf_0}{\pi_0^\perp\bar Z f_0},\\
& = -\g{\pi_0^\perp Zf_0}{\pi_2^\perp\bar Z\pi_0^\perp\bar Z f_0}.
\end{align*}
Therefore if neither $\pi_0^\perp Zf_0$ nor $\pi_0^\perp\bar Z f_0$ is zero, $\caQ$ vanishes if and only if
$\pi_2^\perp\bar Z$ is identically zero on $\ell_2$. Hence $\bar\partial\varphi_2'$ vanishes and $\varphi_2$ is 
holomorphic. 

Finally, we claim that when $\caQ=0$ the induced metric for $\varphi_2$ is $-u_2^2|dz|^2$. To see this, 
let $(f_0,f_1,f_2)$ be a local Toda frame for $f$. 
Then by definition $\varphi_2$ is given locally by the family of lines $[f_2]$. To calculate the differential we use
\eqref{eq:nabla} to deduce that in this frame $d\varphi_2$ is represented by 
\[
\begin{pmatrix}  & 0 & \\ Q/u_1u_2 & & 0\\  & u_2 & \end{pmatrix} dz +
\begin{pmatrix}  &-\bar Q/u_1u_2 & \\ 0 & & u_2\\  & 0 & \end{pmatrix} d\bar z. 
\]
Here we use blank spaces to indicate the Lie subalgebra $\fh_2\subset\fg$ of isotropy group $H_2$: 
relative to the frame $d\varphi_2$ takes values in $\fh_2^\perp$.
Therefore the induced metric is 
\[
-\tfrac12\tr (d\varphi_2)^2 = (\frac{|Q|^2}{u_1^2u_2^2} -u_2^2)|dz|^2.
\]
\end{proof}
Note that, since $\partial\varphi_2':\ell_2\to K\ell_0$, $f$ is the $\partial$-Gauss transform of $\varphi_2$. 

\section{Moduli}\label{sec:moduli}

Theorems \ref{thm:minimal} and \ref{thm:hol} provide parameterisations for different components of the set
\[
\caV = \{(\rho,f):\text{$\rho$ irreducible, $f$ branched minimal}\}/G,
\]
where the quotient is by the simultaneous action of $G$ as conjugation of $\rho$ and ambient isometry of $f$. 
By those theorems it is natural to write $\caV$ as a disjoint union of the sets
\[
\caV(d_1,d_2) = \{(\rho,f)\in\caV:\text{$f$ has $d_1$ anti-complex points and $d_2$ complex points}\}
\]
and 
\[
\caW_\tau = \{(\rho,f)\in\caV: \text{$f$ is $\pm$-holomorphic and}\  \tau(\rho)=\tau\}.
\]

\subsection{The structure of $\caV(d_1,d_2)$.}

As explained at the end of \S 2, each point of $\caV(d_1,d_2)$ is parametrised by a quadruple 
$(\Sigma_c,D_1,D_2,\xi)$. To understand the space of these quadruples we must understand 
how $H^1(\Sigma_c,K^{-2}(D_1+D_2))$ varies with $(\Sigma_c,D_1,D_2)$.  Note that
\[
\deg(K^{-2}(D_1+D_2)) = d_1+d_2-4(g-1)<0,
\]
by the inequalities \eqref{eq:stab}. Whenever a holomorphic line bundle $\caL$ over $\Sigma_c$ has negative degree 
$d$ its first cohomology has, by the Riemann-Roch theorem, dimension
\[
h^1(\caL) = g-1 -d.
\]
Therefore as $\caL$ moves over $\Pic_d(\Sigma_c)$, the Picard component of degree $d$ line bundles, the dimension of
$H^1(\Sigma_c,\caL)$ is constant. Now $\Sigma_c\times\Pic_d(\Sigma_c)$ carries a tautological line bundle
$\caP$ (sometimes called a Poincar\'{e} line bundle) whose fibre over $(p,\caL)$ is the fibre of $\caL$ at $p$. The 
vector spaces $H^1(\Sigma_c,\caL)$ are the fibres of the higher direct image $R^1\pi_*(\caP)$ for the
projection $\pi:\Sigma_c\times \Pic_d(\Sigma_c)\to\Pic_d(\Sigma_c)$ to the second factor. By 
a theorem of Grauert \cite[III, Cor 12.9]{Har} their constant dimension implies
they form a vector bundle over $\Pic_d(\Sigma_c)$. 
In particular, for $d=d_1+d_2-4(g-1)$ this bundle has rank
\[
h^1(K^{-2}(D_1+D_2)) = 5g-5 -d_1-d_2.
\]
The pairs $(D_1,D_2)$ are parametrised by the product of symmetric products $S^{d_1}\Sigma_c\times S^{d_2}\Sigma_c$ (in
which the co-prime pairs occupy an open subvariety).  The bundle can be pulled back along the holomorphic map
\[
S^{d_1}\Sigma_c\times S^{d_2}\Sigma_c\to \Pic_d(\Sigma_c);\quad (D_1,D_2)\mapsto K^{-2}(D_1+D_2),
\]
and the total space of the pullback parametrises the data $(D_1,D_2,\xi)$. It is a connected non-singular 
complex manifold of dimension $5g-5$. 

As $c$ varies over the Teichm\"{u}ller space $\caT_g$ of $\Sigma$ we can take the disjoint union
$\caC_g=\cup_{c\in\caT_g}\Sigma_c$, and likewise for any of the objects $S^d\Sigma_c$ or $\Pic(\Sigma_d)$ above.
In each of these cases we obtain a \emph{complex analytic family over $\caT_g$}, i.e., the total space is a
complex manifold for which the projection onto $\caT_g$ is holomorphic map, and although only a fibre bundle in the 
smooth category the fibre over $c$ is a complex submanifold biholomorphic to the structure determined
by $c$. In particular, 
$\pi_\caC:\caC_g\to\caT_g$ is called the \emph{(universal) Teichm\"{u}ller curve}.
About each point $\caC_g$
has a \emph{permanent uniformising local parameter}, i.e., a complex chart $(\caU,\fz)$ for which
$\fz=(z_1,\ldots,z_{3g-3},\zeta)$ has the properties that: (i) each non-empty intersection $\caU_c=\pi_\caC^{-1}(c)\cap\caU$ 
is such that $(\caU_c,\zeta)$ is a chart on $\Sigma_c$; (ii) the coordinates $z_j$ are constant on the fibres.
The existence of such a chart is an immediate consequence of Bers' construction of $\caC_g$ as a quotient 
of an open submanifold $\caF_g\subset \caT_g\times\C$ by a properly discontinuous action of $\pi_1\Sigma$ which preserves the
fibes over $\caT_g$ \cite{Ber73} (this is also just the standard picture of Kodaira-Spencer for unobstructed deformations of
complex structure \cite{Kod}). 
It follows that one can put complex charts on the symmetric fibre-products of $\caC_g$ over $\caT_g$, whose fibres are
$S^d\Sigma_c$, to obtain non-singular complex analytic families over $\caT_g$. The corresponding families of Picard 
components have been constructed by Earle \cite{Ear}.

Thus for each pair $d_1,d_2$ satisfying \eqref{eq:stab}
we obtain a manifold parametrising the data $(\Sigma_c,D_1,D_2,\xi)$ with $\deg(D_j)=d_j$. Each 
is clearly a connected non-singular complex manifold of dimension $8g-8$. Therefore we have proved the following
theorem.
\begin{thm}
Each set $\caV(d_1,d_2)$ can be given the structure of a non-singular complex manifold of dimension $8g-8$. With this
structure $\caV(d_1,d_2)$ is complex analytic family over Teichm\"{u}ller space $\caT_g$.  The fibre 
over $c\in\caT_g$ is a complex submanifold biholomorphic to a holomorphic vector bundle over 
$S^{d_1}\Sigma_c\times S^{d_2}\Sigma_c$ of rank $5(g-1)-d_1-d_2$.
\end{thm}

\subsection{The structure of $\caW_\tau$}\label{subsec:caW}
With $G=PU(2,1)$ let $\caH(\Sigma_c,G)_\tau$ denote the connected component of Higgs bundles with Toledo invariant $\tau$.
From \S 5 we note that for each $c\in\caT_g$ the subspace $\caW_\tau(c)\subset\caW_\tau$ of maps for which $c$ is fixed
is identifiable with the space of all stable length two Hodge bundles in $\caH(\Sigma_c,G)_\tau$. Since we are only
considering irreducible representations, $|\tau|<2(g-1)$. 
As Gothen \cite{Got} points out, stable length two Hodge bundles can be identified with $\alpha$-stable holomorphic
triples of the form $(V,K^{-1},\Phi_1)$ with $\alpha=2(g-1)$ (and $\deg(V)=-\tfrac32\tau$). There are several places 
in the literature where spaces of such 
objects have been studied, but the most convenient description for our purposes is \cite{BraGPG}. As a corollary of their results
(in particular, \cite[Theorem A]{BraGPG}) we have the following. 
\begin{thm}\label{thm:caW}
The set $\caW_\tau$ has the natural structure of a smooth complex manifold of dimension $9(g-1)-\tfrac32\tau$. The subset
$\caW_\tau(c)$ is a smooth analytic subvariety birational to a $\CP^N$-bundle over $\Jac(\Sigma_c)$, where $N=5(g-1)-\tfrac32\tau-1$.
\end{thm}
\begin{proof}
According to \cite{BraGPG} a holomorphic triple $(E_1,E_2,\varphi)$, where $\varphi:E_2\to E_1$ with $\rk(E_1)=2$ and
$\rk(E_2)=1$, corresponds to a stable $U(2,1)$-Higgs bundle $(E_1\oplus E_2,\varphi)$ when the holomorphic triple is
$\alpha$-stable for $\alpha=2(g-1)$. The moduli space of such triples is, in their notation, $\caN^s_\alpha(2,1,k_1,k_2)$
where $k_j=\deg(E_j)$. For our situation $k_1=-\tfrac32\tau$ and $k_2=2(1-g)$. By \cite[Theorem A]{BraGPG} this is a smooth, 
irreducible variety of dimension $7(g-1)-\tfrac32\tau +1$ and it is birationally equivalent to a $\CP^N$-bundle over 
$\Pic_{k_1-k_2}(\Sigma_c)\times\Pic_{k_2}(\Sigma_c)$ which we will denote $\caN^s_\ell$. This latter space is the 
space of \emph{all} triples $(E_1,E_2,\varphi)$ which can be constructed as extensions of the form
\[
0\to E_2\stackrel{\varphi}{\to} E_1\to F\to 0,
\]
where $F$ is a line bundle of degree $k_1-k_2$. The fibres of $\caN^s_\ell$ are the projective spaces of $H^1(\Sigma_c,
F^{-1}\otimes E_1)$ which parametrise such extensions. The index $\ell$ is sufficiently large that these are all $\ell$-stable and 
$\ell \geq 2(g-1)$ so that all $\alpha$-stable triples are $\ell$-stable. 
The birational map $\caF:\caN^s_\alpha\to \caN^s_\ell$ is the map which recognizes the $\ell$-stability of all
$\alpha$-stable triples which arise from such an extension (in particular, it is defined on the open set
of those triples for which $\varphi$ is a nowhere zero section of $\Hom(E_2,E_1)$). 

Now the abelian variety $\Jac(\Sigma_c)$ acts
freely on triples by $(E_1,E_2,\varphi)\mapsto (E_1\otimes L,E_2\otimes L,\varphi)$ and the birational map $\caF$ is
equivariant for this action.  Triples of the form $(V,K^{-1},\Phi_1)$ are identified with points in the quotient
$\caQ_\alpha(c)=\caN^s_\alpha/\Jac(\Sigma_c)$, which is smooth since the action is free and proper. The 
quotient is therefore birationally equivalent to a $\CP^N$-bundle $\caQ_\ell(c)$ over
\[
(\Pic_{k_2}(\Sigma_c)\times\Pic_{k_1-k_2}(\Sigma_c))/\Jac(\Sigma_c)\simeq \Jac(\Sigma_c).
\]
This bundle has dimension $6(g-1)-\tfrac32\tau$. 
Now, we can identify 
\[
\caW_\tau = \cup_{c\in\caT_c}\caW_\tau(c)
\]
with an analytic subvariety of the ``universal'' $G$-Higgs bundle moduli space $\caH(\caC_g,G))$ for $G=PU(2,1)$. 
Here $\caC_g$ is the universal Teichm\"{u}ller curve over $\caT_g$ and $\caH(\caC_g,G)$ is the complex analytic family
over $\caT_g$ whose fibre at $c$ is $\caH(\Sigma_c,G)$, the moduli space of $G$-Higgs bundles over $\Sigma_c$ (a
construction for $\caH(\caC_g,G)$ can be found in \cite[Thm 7.5]{AleC}). The space $\caW_\tau$ corresponds to the stable
$PU(2,1)$-Higgs bundles with $\Phi_2=0$. This total space has dimension $9(g-1)-\tfrac32\tau$.
\end{proof}
Note that from Theorem \ref{thm:hol} the Toledo invariant places an upper bound on the degree $b$ of the branch divisor $B$ 
of the map $f$: $0\leq b<2(g-1)-\tau$. One can think of $\caW_\tau$ as stratified by subvarieties corresponding to maps
with degree of branching $b$. 
 
\subsection{Map from $\caV$ to $\caR(G)$.}\label{sec:R(G)}
In order to understand when we can use minimal surface data to parametrise representations, we must understand the map
\begin{equation}\label{eq:R}
R:\caV\to\caR(G),\quad (\rho,f)\mapsto \rho.
\end{equation}
We can expect this to be smooth. From the results above, this is likely to be most interesting on the components
$\caV(d_1,d_2)$ since these have the same dimension as $\caR(G)$. While a full understanding of this map will require
further work, we can at least make some interesting comments about its behaviour on the fibres $\caV_c$ of $\caV$ over
Teichm\"{u}ller space. With a fixed conformal structure $c$ we can identify $\caR(G)$ with the moduli space
$\caH(\Sigma_c,G)$ of $G$-Higgs bundles. Then $R$ is injective on $\caV_c$, since it amounts to inclusion (equally, this is a
consequence of the uniqueness theorem for twisted harmonic maps \cite{Cor,Don}). Indeed
\[
\caV_c = \{(E,\Phi)\in \caH(\Sigma_c,G): \tr\Phi^2=0\},
\]
and so it plays the role of the \emph{nilpotent cone} in $\caH(\Sigma_c,G)$. Let us consider the structure of this in light of
the discussion above. Recall that $\|\Phi\|_{L^2}^2$ is a proper Morse-Bott function on $\caH(\Sigma_c,G)$ (at least at
smooth points): we will normalise this by defining
\[
\fF(E,\Phi)= \frac{i}{2}\int_\Sigma\tr(\Phi\wedge\Phi^\dagger). 
\]
Whenever the twisted harmonic map determined by $(E,\Phi)$ is conformal, we have 
\[
\fF(E,\Phi)=\int_\Sigma v_\gamma= \Area_\gamma(\Sigma),
\]
for the induced metric $\gamma$. Now fix a non-maximal value $\tau$ for the Toledo invariant and consider the connected component
$\caH(\Sigma_c,G)_\tau$. Whenever $d_2=\tfrac32\tau +d_1$ this component contains $\caV_c(d_1,d_2)$. Inside the latter lies
the locus $\xi=0$ consisting of length-three Hodge bundles, and this represents one connected critical manifold of $\fF$
(cf.\ \cite[\S 3]{Got}). Since $\xi=0$
exactly when $\caQ=0$ we deduce from Prop.\ \ref{prop:area} that this is the level $\fF=(4g-4-d_1-d_2)\pi$. As we move along
the fibres of the bundle $\caV_c(d_1,d_2)$ Prop.\ \ref{prop:area} tells us that $\fF <(4g-4-d_1-d_2)\pi$. Moreover, a
comparison with \cite[Prop 3.2]{Got} shows that the
dimension of these fibres equals the Morse index of the critical manifold. Indeed, we conjecture that the
bundle $\caV_c(d_1,d_2)$ is precisely the unstable manifold of the
vector field $-\grad\caF$ as a bundle over the critical manifold.  This seems to be a manifestation of Hausel's
theorem \cite[Thm 5.2]{Hau}. He proved that in the moduli space of stable $GL(2,\C)$-Higgs bundles of odd degree, 
the downward Morse
flow coincides with the nilpotent cone. Although Hausel only gave the proof for $GL(2,\C)$, the ingredients hold
equally well in the case of the smooth components of $\caH(\Sigma_c,G)$ for a real form $G$ \cite{Gotprivate}.  

By contrast the image of $\caW_\tau(c)$ in $\caH(\Sigma_c,G)_\tau$ gives the absolute minima of $\caF$ in
that connected component \cite{Got}. On the smooth locus of $\caH(\Sigma_c,G)$, hence particularly when $\tau\neq 2\Z$, there is a single 
connected component of minima for fixed $\tau$ (since $\caF$ is a perfect Morse-Bott function).

\appendix

\section{Uniqueness of minimal embeddings.}\label{sec:unique}

Here we give the proof of Proposition \ref{prop:unique}. We are assuming that $\exp^\perp:TM^\perp\to N$ is a diffeomorphism,
and therefore there is a radial distance function $\rho:N\to \R^+_0$ given by $\rho(p)=\|(\exp^\perp)^{-1}(p)\|$. 
The idea of the proof is to show that the condition on $d v_r/dr$ means each 
$\varphi_r=\exp^\perp(r\eta)$ must have non-zero mean curvature, and that by local comparison every immersion must also 
have non-zero mean curvature at non-zero maximum values of $\rho$. Note that this is a local argument: 
we do not need the existence of global sections of $TM^\perp$ of unit length (which will not, in general, exist).

First recall that for a family of immersions $\varphi_t:M\times\R\to (N,g)$ with 
variational vector field $V = \varphi_*\partial/\partial_t$ a standard calculation gives
\[
\frac{d\vol_\gamma(t)}{dt}|_0 = d\star g(V,d\varphi) - g(V,H_\varphi)\vol_\gamma,
\]
where $H_\varphi=\tr_\gamma\II_\varphi$ is the mean curvature for $\gamma=\varphi^*g$.
In particular, for the mean curvature $H_r$ of the map $\varphi_r$, 
\begin{equation}\label{eq:H}
d\rho(H_r) = g(\grad\rho,H_r) = -\frac{1}{\vol_r}\frac{d\vol_r}{dr},
\end{equation}
since $\grad\rho$ is a normal variation.

Next we show that a local embedding $\varphi$ which comes from an arbitrary non-vanishing local section 
$\nu$ of $TM^\perp$ must have non-zero mean curvature at any point at which $|\nu|$ has a local 
maximum.
\begin{lem}\label{lem:nonmin}
Let $\varphi:\caU\to N$ be an embedding of the form $\varphi=\exp^\perp(\nu)$ for some local 
section $\nu$ of $TM^\perp$ which does not vanish on an open subset $\caU\subset M$,
and suppose $u=\|\nu\|=\rho\circ\varphi$ has a local maximum at $x\in \caU$.  For each $r>0$ set $\varphi_r
= \exp^\perp(r\nu/u)$, and let $\vol_r$ be the volume form for $\varphi_r^*g$. Suppose that
$d\vol_r/dr>0$ at $x$ for each $r$, then $\varphi$ must have non-zero mean curvature $H_\varphi$ at $x$. 
\end{lem}
\begin{proof}
Consider the expressions for the mean curvatures $H_\varphi$ and $H_r$, considered as the 
tension fields $\tau(\varphi)$ and $\tau(\varphi_r)$, in terms the tension fields for $u=\rho\circ\varphi$ and
the constant function $r=\rho\circ\varphi_r$. The composition formulas \cite[2.20]{EelL} give
\begin{eqnarray*}
\tau(u) &=& d\rho(H_\varphi) + \tr_\gamma\nabla d\rho(d\varphi,d\varphi),\\
0=\tau(\rho\circ\varphi_r) &=& d\rho(H_r) + \tr_{\gamma_r}\nabla d\rho(d\varphi_r,d\varphi_r),
\end{eqnarray*}
where $\gamma = \varphi^*g$ and $\gamma_r=\varphi_r^*g$.
Now $\tau(u) = \tr_\gamma\Hess(u)$ and at the local maximum $x$ we have $du=0$, which implies
$d\varphi=d\varphi_{u(x)}$ and $\gamma=\gamma_{u(x)}$.
Therefore at $x$ we have
\[
d\rho(H_\varphi)|_x = \tr_\gamma\Hess(u)|_x + d\rho(H_{u(x)})|_x.
\]
Since $x$ is a local maximum we have $\tr_\gamma\Hess(u)|_x\leq 0$, and 
by assumption $d\rho(H_r)|_x <0$ for all $r>0$, using \eqref{eq:H}. Thus $H_\varphi$ cannot vanish at $x$.
\end{proof}
\begin{proof}[Proof of Prop \ref{prop:unique}]
Suppose $\psi:M\to N$ is any immersion transverse to the fibres of $\exp^\perp$, other than $f$. The function $\rho\circ\psi$ 
must have a non-zero maximum at some $y\in M$. 
Then there is a local section $\nu$ of $TM^\perp$ and a local diffeomorphism $\alpha$ on $M$
for which $\psi =\varphi\circ\alpha$ where $\varphi= \exp^\perp(\nu)$, and 
$H_\psi|_y =  H_\varphi|_{\alpha(y)}$ as elements of $T_{\psi(y)}N$.
Now $(\varphi,\alpha(y))$ satisfy the conditions of Lemma \ref{lem:nonmin}, so $H_\psi|_y\neq 0$.  
\end{proof}

\end{document}